\documentclass{amsart}

\usepackage{amsmath,amsthm,amssymb,amsfonts}
\usepackage[utf8]{inputenc}
\usepackage{hyperref}

\theoremstyle{plain}
\newtheorem{theorem}{Theorem}[section]
\newtheorem{lemma}[theorem]{Lemma}
\newtheorem{claim}[theorem]{Claim}
\newtheorem{corollary}[theorem]{Corollary}
\newtheorem{proposition}[theorem]{Proposition}

\newtheorem{question}[theorem]{Question}
\newcommand{\thistheoremname}{}
\newtheorem*{genericthm*}{\thistheoremname}
\newenvironment{namedthm*}[1]
{\renewcommand{\thistheoremname}{#1}%
  \begin{genericthm*}}
  {\end{genericthm*}}
\theoremstyle{definition}
\newtheorem{definition}[theorem]{Definition}

\newtheorem{remark}[theorem]{Remark}

\numberwithin{equation}{section}

\newcommand{\R}{\mathbb{R}}
\newcommand{\Q}{\mathbb{Q}}
\newcommand{\N}{\mathbb{N}}
\newcommand{\Z}{\mathbb{Z}}
\DeclareMathOperator{\hp}{Homeo^+}
\DeclareMathOperator{\homeo}{Homeo}
\DeclareMathOperator{\sign}{sign}
\DeclareMathOperator{\fix}{Fix}

\DeclareMathOperator{\aut}{Aut}
\DeclareMathOperator{\id}{id}

\begin{document}

\title{The structure of random homeomorphisms}
	\author[U. B. Darji]{Udayan B. Darji}
	\address{Department of Mathematics, University of Louisville,
		Louisville, KY 40292, USA\\Ashoka University, Rajiv Gandhi Education City, Kundli, Rai 131029, India} 
	\email{ubdarj01@louisville.edu}
	\urladdr{http://www.math.louisville.edu/\!$\tilde{}$ \!\!darji}
	\author[M. Elekes]{M\'arton Elekes}
	\address{Alfr\'ed R\'enyi Institute of Mathematics, Hungarian Academy of Sciences,
		PO Box 127, 1364 Budapest, Hungary and E\"otv\"os Lor\'and
		University, Institute of Mathematics, P\'azm\'any P\'eter s. 1/c,
		1117 Budapest, Hungary}
	\email{elekes.marton@renyi.mta.hu}
	\urladdr{www.renyi.hu/ \!$\tilde{}$ \!\!emarci}
	
	\author[K. Kalina]{Kende Kalina}
		\address{E\"otv\"os Lor\'and
		University, Institute of Mathematics, P\'azm\'any P\'eter s. 1/c,
		1117 Budapest, Hungary}
       \email{kkalina@cs.elte.hu} 
	\author[V. Kiss]{Viktor Kiss}
	\address{Alfr\'ed R\'enyi Institute of Mathematics, Hungarian Academy of Sciences,
		PO Box 127, 1364 Budapest, Hungary and E\"otv\"os Lor\'and
		University, Institute of Mathematics, P\'azm\'any P\'eter s. 1/c,
		1117 Budapest, Hungary}
	\email{kiss.viktor@renyi.mta.hu}
	
	\author[Z. Vidny\'anszky]{Zolt\'an Vidny\'anszky}

	\address{Kurt Gödel Research Center for Mathematical Logic, Universität Wien, Währinger Strasse 25, 1090 Wien, Austria and Alfr\'ed 	R\'enyi Institute of Mathematics, Hungarian Academy of Sciences,
		PO Box 127, 1364 Budapest}

	\subjclass[2010]{Primary 03E15, 37E30; Secondary 28A05, 54H11, 28A99, 51F25}
	\keywords{Key Words: non-locally compact Polish group, Haar null, Christensen, shy,
		prevalent, typical element, homeomorphism group, unitary group, conjugacy class} 
        	\email{ zoltan.vidnyanszky@univie.ac.at}
     \urladdr{www.logic.univie.ac.at/$\tilde{}$ vidnyanszz77}
	
	\thanks{The second, fourth and fifth authors were partially supported by the
		National Research, Development and Innovation Office
		-- NKFIH, grants no.~113047, no.~104178 and no.~124749. The fifth author was also partially supported by FWF Grant P29999.}

\begin{abstract}
In order to understand the structure of the ``typical'' element of a homeomorphism group, one has to study how large the conjugacy classes of the group are. When typical means generic in the sense of Baire category, this is well understood, see e.g. the works of Glasner and Weiss, and Kechris and Rosendal. Following Dougherty and Mycielski we investigate the measure theoretic dual of this problem, using Christensen's notion of Haar null sets. When typical means random, that is, almost every with respect to this notion of Haar null sets, the behaviour of the homeomorphisms is entirely different from the generic case. 
For $\hp([0,1])$ we describe the non-Haar null conjugacy classes and also show that their union is co-Haar null, for $\hp(\mathbb{S}^1)$ we describe the non-Haar null conjugacy classes, and for $\mathcal{U}(\ell^2)$ we show that, apart from the classes of the multishifts, all conjugacy classes are Haar null. As an application we affirmatively answer the question whether these groups can be written as the union of a meagre and a Haar null set.
\end{abstract}
\maketitle
\tableofcontents
 
\section{Introduction}

The study of generic elements of Polish groups is a flourishing field with a 
large number of applications, see the works of Kechris and Rosendal \cite{KechrisRosendal}, Truss \cite{truss1992generic}, and Glasner and Weiss \cite{glasner2003universal} among others. It is natural to ask 
whether there exist measure theoretic analogues of these results. 
Unfortunately, on non-locally compact groups there is no natural invariant 
$\sigma$-finite measure. However, a generalisation of the ideal of measure zero 
sets can be defined in every Polish group as follows:

\begin{definition}[Christensen, \cite{originalhaarnull}] Let $G$ be a 
Polish group and $B \subset G$ be Borel. We say that $B$ is \textit{Haar null} if 
there exists a
Borel probability measure $\mu$ on $G$ such that for every $g,h \in G$
we have $\mu(gBh)=0$. 
An arbitrary set $S$ is called Haar null if $S \subset B$ for some Borel Haar 
null set $B$.
\end{definition}

It is known that the collection of Haar null sets forms a $\sigma$-ideal in 
every Polish group and it coincides with the ideal of measure zero sets in 
locally compact groups with respect to every left (or equivalently right) Haar measure. 
Using this definition, it makes sense to talk about the properties of random 
elements of a Polish group. A property $P$ of elements of a Polish group $G$ is said 
to \textit{hold almost surely} or \emph{almost every element of G has property 
$P$} if the set
$\{g \in G: g \text{ has property } P\}$ is co-Haar null.

Since we are primarily interested in homeomorphism and automorphism groups, and in such groups conjugate elements can be considered isomorphic, we are only 
interested in the conjugacy invariant properties of the elements of our Polish groups. Hence, in order to describe  the random element, 
one must give a complete description of the size of the conjugacy classes with 
respect to the Haar null ideal. The investigation of this question has been 
started by Dougherty and Mycielski \cite{DM} in the permutation group of a 
countably infinite set, $S_\infty$. If $f \in S_\infty$ and $a$ is an element 
of the underlying set then the set $\{f^{k}(a):k \in \mathbb{Z}\}$ is called 
the \textit{orbit of $a$ (under $f$)}, while the cardinality of this set is 
called \textit{orbit length}. Thus, each $f \in S_\infty$ has a collection of 
orbits (associated to the elements of the underlying set). It is easy to 
show that two elements of $S_\infty$ are conjugate if and only if they have the 
same (possibly infinite) number of orbits for each possible orbit length. 

\begin{theorem}[Dougherty, Mycielski, \cite{DM}] Almost every element of 
$S_\infty$ has 
infinitely many infinite orbits and only finitely many finite ones.
\end{theorem}
Therefore, almost all permutations belong to the union of a countable set of 
conjugacy classes.
\begin{theorem}[Dougherty, Mycielski, \cite{DM}] 
\label{t:DMold} All of these countably many conjugacy classes are non-Haar null.
\end{theorem}

Thus, the above theorems give a complete description of the non-Haar null 
conjugacy classes and the (conjugacy invariant) properties of a random element. 
In a previous paper \cite{auto}, we studied the conjugacy classes of 
automorphism groups of countable structures. We generalised the construction of 
Dougherty and Mycielski to show that if a closed subgroup $G$ of $S_\infty$ has 
the \emph{finite algebraic closure property} (that is, for every finite $S 
\subset \N$, the set $\{b \in \N : |G_{(S)}(b)| < \infty\}$ is finite, where 
$G_{(S)}(b)$ is the orbit of $b \in \N$ under the stabilizer $G_{(S)}$ of 
$S$) then almost all elements of $G$ have infinitely many infinite and finitely 
many finite orbits. We gave descriptions of the non-Haar null conjugacy classes 
in $\aut(\Q, <)$, the group of order-preserving automorphisms of the rationals and 
in $\aut(\mathcal{R})$, the automorphism group of the random graph. A surprising feature of these descriptions is that in both cases there are continuum many non-Haar null conjugacy classes, hence in what follows the results stating that in certain groups there are only countably many non-Haar null conjugacy classes is non-trivial.

After studying these automorphism groups, the next natural step is to consider homeomorphism groups. Indeed, the division line between these two types of groups is not at all strict, since e.g. it is well known that the automorphism group of the countable atomless Boolean algebra is actually isomorphic to the homeomorphism group of the Cantor set. Our second motivation is that the structure of the generic elements of homeomorphism groups already has a huge literature (see e.g. the monograph by Akin, Hurley and Kennedy \cite{AHK}, or the survey of Glasner and Weiss \cite{GW}) partly because these homeomorphism groups are the central objects in topological dynamics. It is also worth noting that the existence of a generic conjugacy class in the homeomorphism group of the Cantor set was among the main motivations for the seminal paper of Kechris and Rosendal. Therefore it is natural to look at the measure theoretic dual of these questions. The aim of the current paper is to initiate this project by characterising the non-Haar null conjugacy classes in certain important homeomorphism groups.

Let us also mention that there have already been several alternative attempts at investigating random homeomorphisms (with respect to certain natural probability measures), see e.g. Graf, Mauldin and Williams \cite{graf}, and Downarowicz, Mauldin and Warnock \cite{MR1182657}. Our results also imply two theorems of Shi \cite{Shi} concerning the set of fixed-points of the random homeomorphism of $[0, 1]$. 

We now describe our main results. 

Let $\hp([0, 1])$ denote the space of increasing homeomorphisms of the unit 
interval $[0, 1]$ with the topology of uniform convergence. For a homeomorphism 
$f \in \hp([0, 1])$ we let $\fix(f) = \{x \in [0, 1] : f(x) = x\}$. Our theorem 
concerning $\hp([0, 1])$ is the following, which we prove in Section 
\ref{s:homeo [0, 1]}.

\begin{namedthm*}{Theorem \ref{t:homeo [0,1] conjugacy classes}}
  The conjugacy class of $f \in \hp([0, 1])$ is non-Haar null if and 
  only if $ \fix(f)$ has no limit point in the open interval 
  $(0, 1)$, and if $x_0 \in \fix(f) \cap (0, 1)$ then $f(x) - x$ does not have 
  a local extremum point at $x_0$. Moreover, the union of the non-Haar null 
  conjugacy classes is co-Haar null.
\end{namedthm*}
\begin{corollary}
  There are only countably many non-Haar null conjugacy classes in $\hp([0, 
  1])$. 
\end{corollary}

As the definition of Haar null sets suggests, to obtain the "if" part of Theorem \ref{t:homeo [0,1] conjugacy classes} one has to construct a probability measure with certain properties. Our measure is somewhat similar to some of the measures constructed by Graf, Mauldin and Williams \cite{graf}. However, a substantial difference between the two approaches is that proving that a property $P$ holds almost surely (in the sense defined in the Introduction) requires to prove not only that the set $\{f\in \homeo^+([0,1]): f \text{ has property $P$}\}$  has measure one (with respect to a certain probability measure) but that all of its translates have this property as well. This typically requires much more work.

As an application of the above characterisation, we answer the following 
question of the first author for the group $\hp([0, 1])$. 

\begin{question}
  \label{q:decomposition}
  Suppose that $G$ is an uncountable Polish group. Can it be written as the 
  union of a meagre and a Haar null set?
\end{question}

Cohen and Kallman \cite{CK} investigated the above question and developed a 
technique to find Haar null-meagre decompositions. However, their method does 
not work for the group $\hp([0, 1])$ as they have also noted. With the help of 
Theorem \ref{t:homeo [0,1] conjugacy classes}, we have the following corollary. 

\begin{namedthm*}{Corollary \ref{c:homeo [0,1] is Haar null cup meagre}}
  The group $\hp([0, 1])$ can be partitioned into a Haar null and a meagre set.
\end{namedthm*}

\begin{remark} In \cite{auto} we showed that the automorphism group of the countable atomless Boolean algebra can be partitioned into a Haar null and a meagre set, but, as mentioned before, this group is isomorphic to the group of homeomorphisms of the Cantor set, hence we have another important homeomorphism group that can be partitioned into a Haar null and a meagre set.
\end{remark}

We now state two results of Shi \cite{Shi}. Both follow easily from our characterisation of the non-Haar null conjugacy classes of $\hp([0, 1])$, Theorem \ref{t:homeo [0,1] conjugacy classes}.  
\begin{corollary}
  For any function $q \in \hp([0, 1])$, the set of functions in $\hp([0, 1])$ which cross the curve $y = q(x)$ in $(0, 1)$ infinitely many times is neither Haar null, nor co-Haar null.
\end{corollary}
\begin{proof}
  Let $\mathcal{F} = \{f \in \hp([0, 1]) : f(x) = x$ infinitely many times$\}$. Then the set in question is $q \mathcal{F}$, hence it is enough to prove that $\mathcal{F}$ is neither Haar null, nor co-Haar null. And the latter statement follows easily from Theorem \ref{t:homeo [0,1] conjugacy classes}. 
\end{proof}

We remark here that in the paper \cite{Shi}, Shi used a slightly modified definition of Haar nullness. Also, instead of proving that the set in the previous corollary is not Haar null, the author proved that it is not \emph{left} or \emph{right Haar null}. We do not define these notions here, only note that since the set is a translate of a conjugacy invariant set, these notions coincide by Lemma \ref{l:haar null conjugacy}. 

The following corollary was also proved by Shi \cite{Shi}.
\begin{corollary}
  For any function $q \in \hp([0, 1])$, the set of functions in $\hp([0, 1])$ that are equal to $q$ on some non-degenerate subinterval of $[0, 1]$ is Haar null.
\end{corollary}
\begin{proof}
  Notice that the set in question is a translate of $\{f \in \hp([0, 1]) : f(x) = x$ on some non-degenerate interval$\}$, which is a countable union of Haar null sets using Theorem \ref{t:homeo [0,1] conjugacy classes}.
\end{proof}

In Section \ref{s:homeo S^1} we turn to the group $\hp(\mathbb{S}^1)$ of 
order-preserving homeomorphisms of the circle. In what follows, $\tau(f)$ 
denotes the rotation number of a homeomorphism $f \in \hp(\mathbb{S}^1)$. If $f \in 
\hp(\mathbb{S}^1)$ has finitely many periodic points then we call one of its periodic 
points $x_0$ of period $q$ \emph{crossing}, if for a lift $F$ of $f^q$ with 
$x_0 \in \fix(F)$, $x_0$ is not a local extremum point of $F(x) - x$. For the 
formal definition of these notions, see Section \ref{s:homeo S^1}.

\begin{namedthm*}{Theorem \ref{t:haar positive conj classes in homeo s1}}
  The conjugacy class of $f \in \hp(\mathbb{S}^1)$ is non-Haar null if and only if 
  $\tau(f) \in \Q$, $f$ has finitely many periodic points and all of them are 
  crossing. Every such conjugacy class necessarily contains an even number of 
  periodic orbits, and for every rational number $0 \le r < 1$ and 
  positive integer $k$ there is a unique non-Haar null conjugacy class with 
  rotation number $r$ containing $2k$ periodic orbits. Moreover, the union of 
  the non-Haar null conjugacy classes and $\{f \in \hp(\mathbb{S}^1) : \tau(f) \not \in 
  \Q\}$ is co-Haar null. 
\end{namedthm*}

\begin{corollary}
  There are only countably many non-Haar null conjugacy classes in $\hp(\mathbb{S}^1)$. 
\end{corollary}

We could not settle the question whether the union of the non-Haar null 
conjugacy classes is co-Haar null, since the following question remains open.

\begin{question}
  \label{q:S1 irrational rot number}
  Is the set $\{f \in \hp(\mathbb{S}^1) : \tau(f) \not \in \Q\}$ Haar null?
\end{question}

Again, we could use the characterisation to answer Question \ref{q:decomposition} for $\hp(\mathbb{S}^1)$. 

\begin{namedthm*}{Corollary \ref{c:homeo S^1 is Haar null cup meagre}}
  The group $\hp(\mathbb{S}^1)$ can be partitioned into a Haar null and a meagre set.
\end{namedthm*}

In Section \ref{s:U(l^2)} we introduce basic results of spectral theory to 
prove the following theorem about the group $\mathcal{U}(\ell^2)$ of unitary 
transformations of the separable, infinite dimensional Hilbert space. 

\begin{namedthm*}{Theorem \ref{t:ul2main}}
  Let $U \in \mathcal{U}(\ell^2)$ be given with spectral measure $\mu_U$ and multiplicity 
  function $n_U$. If the conjugacy class of $U$ is non-Haar null then $\mu_U   
  \simeq \lambda$ and $n_U$ is constant $\lambda$-a.~e., where $\lambda$ denotes Lebesgue measure. 
\end{namedthm*}
\begin{namedthm*}{Corollary \ref{c:multishift}}
  In the unitary group every conjugacy class is Haar null possibly except for 
  a countable set of classes, namely the conjugacy classes of the multishifts. 
\end{namedthm*}

For the group $\mathcal{U}(\ell^2)$ the fact that it can be written as the union of a meagre and a Haar null set has already been proved, see \cite{CK}. 

\section{Preliminaries and notation}

Throughout the paper $\N$ denotes the set of non-negative integers. Let $G$ be an arbitrary Polish group. We call a set $H \subset G$ 
\emph{conjugacy invariant} if it is the union of some conjugacy classes of $G$. 
A map $\varphi$ with domain $G$ is called \emph{conjugacy invariant} if 
$\varphi(f) = \varphi(g)$ for conjugate elements $f, g \in G$. For an arbitrary 
function $f$, let $\fix(f) = \{x : f(x) = x\}$ denote the set of fixed 
points of $f$.

We will sometimes use the following well-known fact. 

\begin{lemma}
  \label{l:haar null conjugacy}
  Let $G$ be a Polish group and $B \subset G$ be a Borel set that is conjugacy 
  invariant. If there exists a Borel probability measure $\mu$ on $G$ such that 
  for every $g \in G$, $\mu(g B) = 0$ then $B$ is Haar null. 
\end{lemma}
\begin{proof}
  It is enough to show that $\mu$ is a witness measure for $B$, that is, for 
  every $g, h \in G$, $\mu(gBh) = 0$. But we have $gBh = ghh^{-1}Bh \subset 
  ghB$, using at the last step that $B$ is conjugacy invariant, hence $\mu(gBh) 
  \le \mu(ghB) = 0$. 
\end{proof}

\section{Homeomorphisms of the interval}
\label{s:homeo [0, 1]}

In this section we investigate $\hp([0, 1])$, the group of increasing 
homeomorphisms of the unit interval $[0, 1]$ and prove Theorem \ref{t:homeo 
[0,1] conjugacy classes}. The topology on $\hp([0, 1])$ is that of uniform 
convergence, that is generated by the metric $d(f, g) = \|f - g\| = \sup_{x \in 
[0, 1]} |f(x) - g(x)|$ for $f, g \in \hp([0, 1])$. 

With respect to this metric, we denote the open ball centred at the homeomorphism $f \in \hp([0, 
1])$ having radius $\delta > 0$ by $B(f, \delta)$. 
Note that if $f_n \to f$ uniformly for the homeomorphisms $f_n, f \in \hp([0, 
1])$, $n = 0, 1, \dots$, then $f_n^{-1} \to f^{-1}$ uniformly, hence a 
compatible metric for $\hp([0, 1])$ is $d'(f, g) = \|f - g\| + \|f^{-1} - 
g^{-1}\|$.

Let $\sign : \R \to \{-1, 0, 1\}$ be the sign function, i.e., 
\begin{equation*}
\sign(x) = \left\{ 
\begin{array}{cl}
  -1 & \text{if $x < 0$}, \\
  0 & \text{if $x = 0$}, \\
  1 & \text{if $x > 0$}.
\end{array}
\right. 
\end{equation*}
For the characterisation of non-Haar null conjugacy classes of 
$\hp([0, 1])$, we use the following well-known fact. 
\begin{lemma}
  \label{l:conjugacy in [0,1]}
  Two homeomorphisms $f, g \in \hp([0, 1])$ 
  are conjugate if and only if there exists a homeomorphism 
  $h \in \hp([0, 1])$ such that 
  $\sign(f(x) - x) = \sign(g(h(x)) - h(x))$ for every $x \in [0, 1]$. 
\end{lemma}
A similar result about homeomorphisms of $\R$ is shown in 
\cite[Lemma 2.3]{F}, from which the lemma can be proved easily. 
Now we are ready to prove our theorem about the non-Haar null 
conjugacy classes of $\hp([0,1])$. 

\begin{theorem}\label{t:homeo [0,1] conjugacy classes}
  The conjugacy class of $f \in \hp([0, 1])$ is non-Haar null if and 
  only if $ \fix(f)$ has no limit point in the open interval 
  $(0, 1)$, and if $x_0 \in \fix(f) \cap (0, 1)$ then $f(x) - x$ does not have 
  a local extremum point at $x_0$. Moreover, the union of the non-Haar null 
  conjugacy classes is co-Haar null.
\end{theorem}
\begin{proof}
  Let 
  \begin{equation*}
  \begin{split}
    \mathcal{L} = \{&f \in \hp([0, 1]) : 
          \text{$\fix(f)$ has no limit point in $(0, 1)$}\}
    \text{ and } \\
    \mathcal{H} = \{&f \in \hp([0, 1]) : \\ &\text{$f(x) - x$ has no 
    local extremum point at any $x_0 \in \fix(f) \cap (0, 1)$}\}. 
  \end{split}
  \end{equation*}
  We first show that $\mathcal{L}$ and $\mathcal{H}$ are co-Haar null sets, 
  which will prove the ``only if'' part of the theorem. 

  \begin{claim}
    \label{c:no limit of fixed points in (0,1)}
    $\mathcal{L}$ is co-Haar null. 
  \end{claim}
  \begin{proof}
    We first prove that $\mathcal{L} \subset \hp([0, 1])$ is Borel.
      Let 
      \begin{equation*}
      \begin{split}
        \mathcal{L}_{k, l, m} = \{& f \in \hp([0, 1]) : \\ &\exists x_1, 
          \dots x_l \in [1/k, 1 - 1/k]\cap \fix(f) 
          \left(i \neq j \Rightarrow |x_i - x_j| \ge 1/m\right) \}.
      \end{split}
      \end{equation*}
      Now $\mathcal{L}_{k,l,m}$ is closed for every $k, l, m > 0$, since if
      $(f_n) \subset \mathcal{L}_{k, l, m}$ is a sequence converging uniformly to 
      $f$, 
      and $x^n_i \in [1/k, 1 - 1/k]  \cap \fix(f_n)$ for $i = 1, \dots, l$ 
      satisfies $|x^n_i - x^n_j| \ge 1/m$ for $i \neq j$, then, for an 
      appropriate subsequence, the numbers $x^n_i$ converge to some $x_i$ for 
      all 
      $i$. It is easy to see that $x_i \in [1/k, 1 - 1/k] \cap \fix(f)$ and 
      $|x_i 
      - x_j| \ge 1/m$ for $i \neq j$. Hence $f \in \mathcal{L}_{k, l, m}$, 
      showing 
      that $\mathcal{L}_{k, l, m}$ is closed. 
      
      Since $f \not \in \mathcal{L}$ if and only if $f$ has infinitely many 
      fixed 
      points in some interval $[1/k, 1 - 1/k]$, one can easily check that 
      $$
        \mathcal{L}^c = \bigcup_{k > 0} \bigcap_{l > 0} \bigcup_{m > 0}
        \mathcal{L}_{k, l, m}.
      $$
      This shows that $\mathcal{L}$ is Borel.
    
    Now as $\mathcal{L}$ is Borel and clearly conjugacy invariant, using Lemma 
    \ref{l:haar null conjugacy} it is enough to find a Borel probability 
    measure 
    $\mu$ with $\mu(\mathcal{L} g) = 1$ for every $g \in \hp([0, 
    1])$. 
    
    For $1/4 \le a \le 3/4$ let $f_a \in \hp([0, 1])$ be the piecewise linear 
    function defined by 
    \begin{equation}
      \label{e:linear witness}
      f_a(x) = \left\{ 
      \begin{array}{cl}
        2xa & \text{if $0 \le x < \frac{1}{2}$}, \\
        2(1 - a)x + 2a - 1 & \text{if $\frac{1}{2} \le x \le 1$}.
      \end{array}
      \right. 
    \end{equation}
    The map $\Phi : a \mapsto f_a$ is continuous from $[1/4, 3/4]$ to $\hp([0, 
    1])$, hence the pushforward of the normalized Lebesgue measure on $[1/4, 
    3/4]$ is a Borel probability measure on $\hp([0, 1])$. Let the pushforward 
    be 
    $\mu$, i.e., 
    \begin{equation}
      \label{e:mu}
      \mu(\mathcal{B}) = 2\lambda(\Phi^{-1}(\mathcal{B})) = 
        2\lambda(\{a : f_a \in \mathcal{B}\})
    \end{equation}
    for every Borel subset $\mathcal{B} 
    \subset \hp([0, 1])$, where $\lambda$ is the Lebesgue measure. 
    
    Let $g \in \hp([0, 1])$ be arbitrary, it is enough to show that 
    $\mu(\mathcal{L} g) = 1$. It is easy to see that 
    $$
      \mu(\mathcal{L} g) = 2\lambda(\{a : \forall k(f_a(x) = g(x) 
      \text{ for only finite many $x$ in $[1/k, 1 - 1/k]$})\}).
    $$
    Now we proceed indirectly. Suppose that $\mu(\mathcal{L} g) < 1$, 
    that is, the set of those $a$ such that $g$ intersects the graph of $f_a$ 
    infinitely many times in some interval $[c_a, 1 - c_a]$ is of positive (outer) 
    measure. Hence, we can find a common $0 < c < 1/2$ such that the set of 
    those $a$ such that $g$ intersects the graph of $f_a$ infinitely many times in 
    the interval $[c, 1 - c]$ is still of positive outer measure. 
    
    Having found this common $c$, we make some modifications to the 
    functions. For a homeomorphism $h \in \hp([0, 1])$, let 
    $h^{*} : [c, 1 - c] \to \R$ be the function defined by 
    \begin{equation*}
      h^{*}(x) = \left\{ 
      \begin{array}{cl}
        \frac{h(x)}{2x} & \text{if $c \le x < \frac{1}{2}$}, \\
        \frac{h(x)}{2(1 - x)} + \frac{1 - 2x}{2(1 - x)} & 
          \text{if $\frac{1}{2} \le x \le 1 - c$}.
      \end{array}
      \right. 
    \end{equation*}
    It is easy to see that $f_a^{*}(x) = a$ for every $a \in [1/4, 3/4]$ and 
    $x \in [c, 1 - c]$. 
    
    Note that the maps $(x, y) \mapsto (x, \frac{y}{2x})$ and $(x, y) \mapsto 
    (x, 
    \frac{y}{2(1 - x)} + \frac{1 - 2x}{2(1 - x)})$ are Lipschitz maps from 
    the region bounded by the graphs of $f_{1/4}|_{[c, 1 - c]}$ and $f_{3/4}|_{[c, 
    1 - c]}$ to $[c, 1 - c] \times 
    [1/4, 3/4]$. This can be shown by checking that the derivatives are 
    bounded. 
    Moreover, the two maps agree on $\{1/2\} \times [1/4, 3/4]$, hence the map 
    \begin{equation*}
      F(x, y) = \left\{ 
      \begin{array}{cl}
        \left(x, \frac{y}{2x}\right) & \text{if $c \le x < \frac{1}{2}$}, \\
        \left(x, \frac{y}{2(1 - x)} + \frac{1 - 2x}{2(1 - x)}\right) & 
          \text{if $\frac{1}{2} \le x \le 1 - c$}
      \end{array}
      \right. 
    \end{equation*}
    is Lipschitz on the same region. This implies that the image of a 
    rectifiable curve under $F$ is rectifiable, hence $g^{*}$ is of bounded 
    variation, as its graph is the image of the graph of $g|_{[c, 1- c]}$ under 
    $F$. 
    
    Now we finish the proof of the theorem by deducing a contradiction. We 
    derived from our indirect assumption that the set of those $a$ such that 
    the graph of $g$ intersects the graph of $f_a$ infinitely many times over 
    the 
    interval $[c, 1 - c]$ is of positive outer measure. This implies that the set of 
    those $a$ such that the graph of $g^{*}$ intersects the graph of the 
    constant 
    function $f_a^{*}$ infinitely many times is of positive outer measure. Hence, the 
    set of those $y$ such that $(g^{*})^{-1}(y)$ is infinite is of positive 
    outer measure. But $g^{*}$ is of bounded variation, hence according to a result 
    of 
    Banach (see \cite[Corollary 1]{B}), the measure of this set should be zero, 
    a contradiction. 
  \end{proof}

  \begin{claim}
    \label{c:no local extremum fixed point is co-Haar null}
    $\mathcal{H}$ is co-Haar null. 
  \end{claim}
  \begin{proof}
    First we show that $\mathcal{H} \subset \hp([0, 1])$ is Borel. It is easy 
    too see that the complement of $\mathcal{H}$ is $\bigcup_{m = 1}^\infty 
    \mathcal{H}_m$, where
    \begin{equation*}
    \begin{split}
      \mathcal{H}_m = \big\{f \in \hp([0, 1]) :\; &\exists x_0 \in [1/m, 1 
      - 
      1/m] \text{ such that $x_0$ is the } \\
      &\text{minimum or the maximum of 
        $f - \id|_{(x_0 - \frac{1}{m}, \; x_0 + \frac{1}{m})}$}\big\}.
    \end{split}
    \end{equation*}
    We show that each $\mathcal{H}_m$ is closed, from which it follows that 
    $\mathcal{H}$ is $G_\delta$, thus Borel. Let $(f_n)_{n \in \N}$ be a 
    convergent sequence of functions from $\mathcal{H}_m$. Let $x_n \in [1/m, 1 
    - 1/m]$ be a point with the property that it is the minimum or the maximum 
    of $f_n|_{(x_n - 1 / m,\; x_n + 1/m)}$. Then there exists a subsequence 
    $n_k$ such that $x_{n_k}$ converges to some $x' \in [1/m, 1 - 1/m]$ with 
    $x_{n_k}$ always either a local minimum or a local maximum of $f_{n_k}$. It 
    is easy to see that if $f$ is the limit of the sequence $(f_n)_{n \in \N}$ 
    then $x'$ is the minimum or maximum of $f|_{(x' - 1/m,\; x' + 1/m)}$, 
    showing that $\mathcal{H}_m$ is closed. This finishes the proof that 
    $\mathcal{H}$ is Borel.
    
    Now we construct a Borel probability measure on $\hp([0, 1])$, such that 
    for every $g \in \hp([0, 1])$, $\mu(g\mathcal{H}) = 1$. 
    By Lemma \ref{l:haar null conjugacy}, this indeed implies that 
    $\mathcal{H}$ is co-Haar null, as it can easily be seen using 
    Lemma \ref{l:conjugacy in [0,1]} that $\mathcal{H}$ is conjugacy 
    invariant. 
    
    This measure is the same as in the proof of Claim \ref{c:no limit of 
    fixed points in (0,1)}, for $1/4 \le a \le 3/4$ let $f_a \in \hp([0, 1])$ 
    be the piecewise linear function as in \eqref{e:linear witness}. 
    The measure $\mu$ is defined as in \eqref{e:mu}, $\mu(\mathcal{B}) = 
    2\lambda(\{a : f_a \in \mathcal{B}\})$ for a Borel set $B$,   
    where $\lambda$ is the Lebesgue measure. As before, this defines a Borel 
    probability measure on $\hp([0, 1])$. 
    Let $g \in \hp([0, 1])$, to finish the proof that $\mathcal{H}$ is 
    co-Haar null, it is enough to show that the set 
    $\{a : f_a \not\in g\mathcal{H}\} = \{a : g^{-1}f_a \not\in \mathcal{H}\}$ 
    is countable. 
    
    If for some $a \in [1/4, 3/4]$, the function $g^{-1}f_a(x) - x$ has 
    a local maximum at $x_a \in \fix(g^{-1}f) \cap (0, 1)$, then we 
    assign an open interval $(c_a, d_a)$ to $a$, where 
    $c_a, d_a \in \Q \cap (0, 1)$, $x_a \in (c_a, d_a)$ and 
    $g^{-1}(f_a(x)) - x \le 0$ if $x \in (c_a, d_a)$. 
    Such an interval exists, as $x_a \in \fix(g^{-1}f)$, 
    hence $g^{-1}(f(x_a)) - x_a = 0$, and $x_a$ is the local maximum of 
    $g^{-1}f_a(x) - x$. 
    Now we show that we cannot assign the same interval to two distinct 
    points. Suppose that we assign $(c_a, d_a)$ to $a$ and $a'$, with 
    $a < a'$. Then for every $x \in (c_a, d_a)$, we have 
    $g^{-1}(f_{a'}(x)) - x \le 0$, but for $x_{a} \in (c_a, d_a)$, 
    $0 = g^{-1}(f_{a}(x_{a})) - x_{a} < g^{-1}(f_{a'}(x_a)) - x_a$, a 
    contradiction, using the fact that $g^{-1}$ is a strictly increasing 
    function, and $f_a(x) < f_{a'}(x)$ for every $x \in (0, 1)$. 
    
    This show that the set of those $a$'s, such that 
    $g^{-1}f_a \not \in \mathcal{H}$ because there is a fixed point 
    in $(0, 1)$ which is a local maximum of $g^{-1}(f_a(x)) - x$, is 
    countable. The other reason excluding $g^{-1}f_a$ from being in 
    $\mathcal{H}$ is the existence of a fixed point that is a local 
    minimum, and one can similarly prove that this happens for only countably 
    many $a$'s. This shows that $\{a : f_a \not\in g\mathcal{H}\}$ is 
    countable, hence $\mathcal{H}$ is co-Haar null. 
  \end{proof}

  Thus the proof of the ``only if'' part of the theorem is complete. 
  
  Now we prove that every conjugacy class described in the 
  theorem is indeed non-Haar null. Let $\mathcal{C}$ be such a conjugacy class, 
  then for those fixed points of a function $f \in \mathcal{C}$ that fall into 
  the 
  open interval $(0, 1)$ we have the following possibilities:
  \begin{itemize}
    \item[(i)] $\fix(f) \cap (0, 1)$ is a countably infinite set with 
      limit points $0$ and $1$, 
    \item[(ii)] $\fix(f) \cap (0, 1)$ is a countably infinite set with 
      the unique limit point $0$,
    \item[(iii)] $\fix(f) \cap (0, 1)$ is a countably infinite set with 
      the unique limit point $1$, 
    \item[(iv)] $|\fix(f) \cap (0, 1)| = n$ for some $n \in \N$.  
  \end{itemize}
  The set $(0, 1) \setminus \fix(f)$ is the union countably many 
  open intervals, and since every fixed point 
  $x_0 \in \fix(f) \cap (0, 1)$ is isolated, each of these intervals 
  has a neighbouring interval on each side. 
  In each of these intervals, the function $f(x) - x$ is 
  either positive or negative, but as $f(x) - x$ has no local extremum point at 
  any fixed point, the sign of $f(x) - x$ is distinct in neighbouring intervals. 
  
  Hence, each of the last three cases in the list describes two conjugacy 
  classes (the forth case describes two conjugacy classes for each $n$): 
  there is a first or last open interval of 
  $(0, 1) \setminus \fix(f)$, in which the sign of $f(x) - x$ 
  can be chosen to be either positive or negative, but this 
  choice determines the sign on every other interval. 
  Whereas the first case describes only one possible conjugacy class. 
  One can easily see using Lemma \ref{l:conjugacy in [0,1]} that these 
  possibilities indeed describe conjugacy classes. 

  Now we prove that in case (ii), (iii) and (iv), in order to show that each of 
  these classes are non-Haar null, it is enough to show that the 
  union of the two described conjugacy classes is non-Haar null. 
  The proof is the same in each case, 
  so let $\mathcal{C}_1$ and $\mathcal{C}_2$ be the two conjugacy classes in 
  either case and suppose that $\mathcal{C}_1$ is Haar null. We show that then 
  $\mathcal{C}_2$ is Haar null, hence so is the union, contradicting our 
  assumption. We use the following claim for the proof. 
  \begin{claim}
    \label{c:invariance under inverse}
    Let $G$ be a Polish topological group and $H \subset G$ a Haar null set. 
    Then the set $\{h^{-1} : h \in H\}$ is also Haar null. 
  \end{claim}
  \begin{proof}
    Let $\mu$ be a witness measure for $H$, and let $\mu'$ be the measure with 
    $\mu'(B) = \mu(\{b^{-1} : b \in B\})$. Now $\mu'$ is a witness 
    measure for the set $H^{-1} = \{h^{-1} : h \in H\}$, since 
    $$
    \mu'(aH^{-1}b) = \mu'(\{ah^{-1}b : h \in H\}) = 
    \mu(\{b^{-1}ha^{-1} : h \in H\}) = 0.
    $$
  \end{proof}
  Using the claim, the image of $\mathcal{C}_1$ under the homeomorphism $f 
  \mapsto f^{-1}$ is also Haar null. But it is easy to see that the image is 
  $\mathcal{C}_2$ (regardless of which case we are working with), hence 
  $\mathcal{C}_2$ is also Haar null, and so is $\mathcal{C}_1 \cup 
  \mathcal{C}_2$. 
    
  Now we state three lemmas that will be essential for the rest of the proof. 
  \begin{lemma}
    \label{l:arzela--ascoli}
    Let $\mathcal{K} \subset \hp([0, 1])$ be a compact set. Then there exists 
    $\delta : (0, \infty) \to (0, \infty)$ such that for every  
    $\varepsilon > 0$, $f \in \mathcal{K}$ and $x, y \in [0, 1]$ with $|x - y| 
    \le \delta(\varepsilon)$, we have $|f(x) - f(y)| < \varepsilon$ and 
    $|f^{-1}(x) - f^{-1}(y)| < \varepsilon$. 
  \end{lemma}
  \begin{proof}
    Let $\mathcal{K}' = \mathcal{K} \cup \{f^{-1} : f \in \mathcal{K}\}$. Then 
    $\mathcal{K}'$ is compact, since the map $f \mapsto f^{-1}$ is a 
    homeomorphism.  
    As $\hp([0, 1])$ is a subspace of $C([0, 1])$, the space of continuous 
    function defined on $[0, 1]$ with the supremum norm, we can apply the 
    Arzel\`a--Ascoli theorem for $\mathcal{K}'$. Hence, there exists a function 
    $\delta : (0, \infty) \to (0, \infty)$ such that for all $f \in 
    \mathcal{K}'$ and $x, y \in [0, 1]$ with $|x - y| \le \delta(\varepsilon)$, 
    we have $|f(x) - f(y)| < \varepsilon$. As $\mathcal{K}'$ contains $f$ and 
    $f^{-1}$ for every $f \in \mathcal{K}$, we are done.
  \end{proof}
  
  \begin{lemma}
    \label{l:cikk-cakk}
    Let $\mathcal{K} \subset \hp([0, 1])$ be compact and $x_0, y_0 \in (0, 
    1)$, then there exists $g \in \hp([0, 1])$ with $g(x_0) = y_0$ and 
    such that both $0$ and $1$ are among the limit points of 
    $\fix(g^{-1}f)$ for every $f \in \mathcal{K}$. 
  \end{lemma}
  \begin{proof}
    Let $\delta : (0, \infty) \to (0, \infty)$ be the function given by the 
    previous lemma for $\mathcal{K}$. Let $x_1 = x_0 + (1 - x_0) /2 < 
    1$, we show that there exists $y_1 < 1$ such that $y_1 > y_0$ and $y_1 > 
    f(x_1)$ for every $f \in \mathcal{K}$. Indeed, any $y_1$ with $\max\{y_0, 1 
    - \delta(1 - x_1)\} < y_1 < 1$ works, since if $|1 - y_1| < \delta(1 - 
    x_1)$ then $|f^{-1}(1) - f^{-1}(y_1)| < 1 - x_1$, showing that $f^{-1}(y_1) 
    > x_1$, hence $f(x_1) < y_1$ for every $f \in \mathcal{K}$. 
    Then let $y_2 = y_1 + (1- y_1) /2$ and $x_2 < 1$ be such that $x_2 > 
    x_1$ and $f(x_2) > y_2$ for every $f \in \mathcal{K}$. One can easily get 
    an appropriate $x_2$ by a similar argument, choosing $x_2$ such that 
    $\max\{x_1, 1 - \delta(1 - y_2)\} < x_2 < 1$. 
    
    By iterating the argument, one can easily get two strictly increasing 
    sequences $(x_n)_{n \in \N}$ and $(y_n)_{n \in \N}$ such that they are 
    converging to 1, and $f(x_{2n}) > y_{2n}$, $f(x_{2n + 1}) < y_{2n + 1}$ for 
    every $f \in \mathcal{K}$ and $n \in \N \setminus \{0\}$. 
    Moreover, similarly one can get two decreasing sequences 
    $(z_n)_{n \in \N}$ and $(u_n)_{n \in \N}$ with $z_0 < x_0$ and $u_0 < y_0$ 
    such that they are converging to 0, and $f(z_{2n}) < u_{2n}$, 
    $f(z_{2n + 1}) > u_{2n + 1}$ for every $f \in \mathcal{K}$ and $n \in \N$. 
    
    One can easily see that $(0, 1) = (z_0, x_0) \cup \bigcup_{n \in \N} (z_{n 
    + 1}, z_n] \cup [x_n, x_{n + 1})$. Moreover, by defining the function $g : 
    [0, 1] \to [0, 1]$ such that $g(x_n) = y_n$ and $g(z_n) = u_n$, for every 
    $n \in \N$ and $g$ is linear on intervals of the form $[x_n, x_{n + 1}]$, 
    $[z_{n + 1}, z_n]$ and on $[z_0, x_0]$ with $g(0) = 0$ and $g(1) = 1$, one 
    gets an increasing homeomorphism of $[0, 1]$ such that the graph of $g$ 
    intersects the graph of every $f \in \mathcal{K}$ infinitely many times in 
    every neighbourhood of $0$ and $1$. Thus the proof of the lemma is complete. 
  \end{proof}
  
  \begin{lemma}
    \label{l:extension of g}
    Let $0< x_0, y_0 < 1$ with $f(x_0) > y_0$ for every $f \in \mathcal{K}$. 
    Then there exists a function $g \in \hp([0, 1])$ such that $g(x_0) = y_0$ 
    and $f(x) > g(x)$ for every $x \in (0, 1)$ and $f \in \mathcal{K}$. 
  \end{lemma}
  \begin{proof}
    For $x \in [0, 1]$, let $h(x) = \min_{f \in \mathcal{K}} f(x)$. The minimum 
    exists, since the image of $\mathcal{K}$ is compact in $[0, 1]$ under the 
    continuous map $f \mapsto f(x)$. Using \cite[Theorem 2.2]{ST}, $h \in 
    \hp([0, 1])$. Let $T : [0, 1] \to [0, 1]$ be the 
    function with $T(0) = T(x_0) = \frac{y_0}{h(x_0)}$, $T(1) = 1$  
    and $T$ is linear in the intervals $[0, x_0]$ and $[x_0, 1]$. Let 
    $$
    g(x) = T(x) h(x),
    $$
    we show that $g$ satisfies the conditions of the lemma. It is easy to see 
    that $g(0) = 0$, $g(1) = 1$ and $g(x_0) = y_0$, and also that it is 
    continuous. Since $T$ is a positive, increasing function, for $0 \le x < y 
    \le 1$ we also have 
    \begin{equation*}
    \begin{split}
      g(y) - g(x) = T(y) h(y) - T(x) h(x) \ge 
      T(x) (h(y) - h(x)) > 0, 
    \end{split}
    \end{equation*}
    showing that $g$ is strictly increasing, hence $g \in \hp([0, 1])$. For 
    $x \in (0, 1)$, $T(x) < 1$, hence $g(x) < h(x) \le f(x)$ for every $f \in 
    \mathcal{K}$. This finishes the proof of the lemma. 
  \end{proof}
  
  Now we prove that the conjugacy classes described by (i), (ii), (iii) and 
  (iv) are indeed non-Haar null. We first deal with the cases (i), (ii) and 
  (iii), as they are easy consequences of the above facts. As already discussed above, cases (ii) and (iii) each describe two conjugacy classes, and in order to prove that both are non-Haar null, it is enough to show that their union is non-Haar null. 
    
  So let 
  \begin{equation*}
  \begin{split}
    \mathcal{F}_{(i)} &= \{f \in \hp([0, 1]) : 
    \text{$0$ and $1$ is among the limit points of $\fix(f)$}\}, \\
    \mathcal{F}_{(ii)} &= \{f \in \hp([0, 1]) : 
    \text{$0$ is a limit point of $\fix(f)$, but $1$ is not}\}, \\
    \mathcal{F}_{(iii)} &= \{f \in \hp([0, 1]) : 
    \text{$1$ is a limit point of $\fix(f)$, but $0$ is not}\}. 
  \end{split}
  \end{equation*}
  Since the conjugacy class described by (i) is $\mathcal{F}_{(i)} \cap 
  \mathcal{H} \cap \mathcal{L}$, and $\mathcal{H}$ and $\mathcal{L}$ are 
  co-Haar null, it is enough to prove that $\mathcal{F}_{(i)}$ is non-Haar null 
  to finish this case. Similarly, in cases (ii) and (iii), the union of the two 
  conjugacy classes are $\mathcal{F}_{(ii)} \cap \mathcal{H} \cap \mathcal{L}$ 
  and $\mathcal{F}_{(iii)} \cap \mathcal{H} \cap \mathcal{L}$, hence for these 
  cases it is enough to show that $\mathcal{F}_{(ii)}$ and 
  $\mathcal{F}_{(iii)}$ are non-Haar null. 

  To show that these three sets are non-Haar null, it is enough to prove that 
  every compact set can be translated into them, as above. So let 
  $\mathcal{K} \subset \hp([0, 1])$ be compact. Let $x_0 \in (0, 1)$ be such 
  that $f(x_0) > 1/2$ for every $f \in \mathcal{K}$, for example $x_0 = 1 - 
  \delta(1/2)$ works, where $\delta$ is the function provided by Lemma 
  \ref{l:arzela--ascoli}. To show that $\mathcal{K}$ can be translated into 
  $\mathcal{F}_{(i)}$, let $g$ be the function provided by Lemma 
  \ref{l:cikk-cakk} with $x_0$ and $y_0 = 1/2$. It is clear that $g^{-1} 
  \mathcal{K} \subset \mathcal{F}_{(i)}$. 
  
  Now let $h$ be the function provided by Lemma \ref{l:extension of g} with 
  $x_0$ as above and $y_0 = 1/2$, and let $g_{(ii)}$ agree with $g$ on $[0, 
  x_0]$, and agree with $h$ on $[x_0, 1]$. Then $g_{(ii)}$ intersects the graph 
  of every $f \in \mathcal{K}$ infinitely many times in every neighbourhood of 
  $0$, but is disjoint from the graph of every $f \in \mathcal{K}$ in a small 
  deleted neighbourhood of $1$. From this one can easily see that $g_{(ii)}^{-1} 
  \mathcal{K} \subset \mathcal{F}_{(ii)}$. Similarly, if $g_{(iii)}$ agrees 
  with $g$ on $[x_0, 1]$ and agrees with $h$ on $[0, x_0]$, then 
  $g_{(iii)}^{-1}\mathcal{K} \subset \mathcal{F}_{(iii)}$, showing that each 
  of these three sets are non-Haar null. This finishes the proof that the 
  conjugacy classes described in (i), (ii) and (iii) are non-Haar null. 
  
  It remains to be proven that for each $n$, the two conjugacy classes 
  described by 
  (iv) are non-Haar null. 
  
  Let 
  $$
    \mathcal{F}_n = \{f \in \hp([0, 1]) : |\fix(f)| = n + 2\}.
  $$
  We claim that it is enough to prove that for each $n \in \N$, $\mathcal{F}_n$ 
  is non-Haar null. For a fixed $n$, if $\mathcal{F}_n$ is non-Haar null, then 
  the union of the two conjugacy classes described by (iv) is also non-Haar 
  null, since the union equals $\mathcal{F}_n \cap \mathcal{H} \cap 
  \mathcal{L}$, the intersection of a non-Haar null and two co-Haar null sets. 
  And as noted before, this implies that both of these classes are non-Haar 
  null. 
    
  First we show that for each $n \in \N$, the set 
  $$
    \mathcal{F}_{\ge n}^{fin} = 
    \{f \in \hp([0, 1]) : n + 2 \le |\fix(f)| < \aleph_0\}
  $$
  is non-Haar null. In order to show this, it is enough to prove that the set 
  \begin{equation*}
  \begin{split}
    \mathcal{F}'_{\ge n} = \{f \in \hp([0, 1]) : \text{$n + 2 \le |\fix(f)|$ 
    and neither $0$ nor $1$} \\ \text{is among the limit points of $\fix(f)$}\}
  \end{split}
  \end{equation*}
  is non-Haar null, since $\mathcal{F}_{\ge n}^{fin} = \mathcal{F}'_{\ge n} 
  \cap \mathcal{L}$, thus then $\mathcal{F}_{\ge n}^{fin}$ is the 
  intersection of a non-Haar null and a co-Haar null set. 
  
  We prove that $\mathcal{F}'_{\ge n}$ is non-Haar null by showing that every 
  compact subset of $\hp([0, 1])$ can be translated into it. So let 
  $\mathcal{K} \subset \hp([0, 1])$ be compact. Similarly as before, there 
  exist $x_0, y_0 \in (0, 1)$ with $f(x_0) > y_0$ for every $f \in 
  \mathcal{K}$. Then, again with the same ideas as in the proof of Lemma 
  \ref{l:cikk-cakk}, one can find two 
  strictly increasing finite sequences in $(0, 1)$, $(x_k)_{k = 0}^{2n}$ and 
  $(y_k)_{k = 0}^{2n}$ such that $f(x_{2m}) > y_{2m}$, $f(x_{2m + 1}) < y_{2m + 
  1}$ for every $f \in \mathcal{K}$ and every suitable $m$. Then let $g$ be 
  defined on the interval $[x_0, x_{2n}]$ by $g(x_k) = y_k$ for every $k \le 
  2n$ and $g$ is linear on each interval $[x_k, x_{k + 1}]$ for every $k < 
  2n$. Then $g$ is a strictly increasing continuous function on $[x_0, 
  x_{2n}]$. Define $g$ on the remaining intervals of $[0, 1]$ with the help of 
  Lemma \ref{l:extension of g} such that $g(x) < f(x)$ for every $x \in (0, 
  x_0) \cup (x_{2n}, 1)$ and $f \in \mathcal{K}$. Then $g \in \hp([0, 1])$, and 
  clearly $g^{-1} \mathcal{K} \subset \mathcal{F}'_{\ge n}$, since $g$ 
  intersects every $f \in \mathcal{K}$ at least $n$ times (actually at least 
  $2n$ times), but there is no intersection in the intervals $(0, x_0)$ and 
  $(x_{2n}, 1)$. This shows that $\mathcal{F}'_{\ge n}$ and thus also that 
  $\mathcal{F}_{\ge n}^{fin}$ is non-Haar null. 
  
  Using that $\mathcal{F}_{\ge n}^{fin} = \bigcup_{m \ge n} \mathcal{F}_m$, and 
  the fact that the countable union of Haar null sets is Haar null, we conclude 
  that at least one of the $\mathcal{F}_m$'s must be non-Haar null, hence for 
  each $n \in \N$ there exists $m \ge n$ such that $\mathcal{F}_m$ is non-Haar 
  null. It follows that to show that each $\mathcal{F}_n$ is non-Haar null, 
  it is enough to prove the following lemma: 
  \begin{lemma}
    For a fixed $n$ if $\mathcal{F}_{n + 1}$ is non-Haar null then 
    $\mathcal{F}_n$ is non-Haar null.
  \end{lemma}
  \begin{proof}
    First we prove that $\mathcal{F}_k$ is Borel for every $k \in \N$. 
    Let 
    $$
      \mathcal{F}_{\ge n} = \{f \in \hp([0, 1]) : n + 2 \le |\fix(f)|\},
    $$
    then $\mathcal{F}_k = \mathcal{F}_{\ge k} \setminus \mathcal{F}_{\ge k + 
    1}$, hence it is enough to show that $\mathcal{F}_{\ge k}$ is Borel for 
    every $k$. 
    Let
    $$
      \mathcal{F}_{\ge k}^\delta = \{f : \exists \; x_1, \dots, x_{k + 2} \in 
      [0, 1] \; \forall i \neq j (f(x_i) = x_i \land |x_i - x_j| \ge \delta)\}.
    $$
    It is easy to see that $\mathcal{F}_{\ge k} = \bigcup_{m \in \N} 
    \mathcal{F}_{\ge k}^{1/m}$, and also that $\mathcal{F}_{\ge k}^{\delta}$ is 
    closed, hence $\mathcal{F}_{\ge k}$ is $F_\sigma$, thus Borel. 
    
    Now, suppose towards a contradiction that $\mathcal{F}_n$ is Haar null. 
    Then there exists a Borel probability measure $\mu$ such that every two-sided
    translate of $\mathcal{F}_n$ is of measure 0 with respect to $\mu$. Using 
    \cite[17.11]{kechrisbook}, there is a compact set $\mathcal{K}$ with 
    $\mu(\mathcal{K}) > 0$, hence by restricting $\mu$ to $\mathcal{K}$ and 
    multiplying accordingly to get a probability measure, we again obtain a 
    measure with the property that every translate of $\mathcal{F}_n$ is of 
    measure 0. Hence, we can suppose that $\mu$ has compact support, which we 
    do from now on. 
    Since $\mathcal{F}_{n + 1}$ is not Haar null, $\mu$ cannot be a witness for 
    it, hence there exists $g \in \hp([0, 1])$ such that $\mu(g \mathcal{F}_{n 
    + 1}) > 0$, where we used Lemma \ref{l:haar null conjugacy} for the fact 
    that it is enough to consider one sided translates of 
    $\mathcal{F}_{n + 1}$. 
    
    Let us fix three rational numbers $p$, $q$ and $r$ with $0 < p < q < r < 
    1$. For a function $f \in \mathcal{F}_{n + 1}$, the first, i.e., the 
    smallest fixed point of $f$ is 0. Let $\mathcal{F}_{n + 1}^{p, q, r}$ 
    contain those $f$'s, such that the second fixed point of $f$ is in $(p, 
    q)$, while the other $n + 1$ fixed points are in $(r, 1]$, that is 
    $$
      \mathcal{F}_{n + 1}^{p, q, r} = \{f \in \mathcal{F}_{n + 1} : 
      |\fix(f) \cap (p, q)| = 1, |\fix(f) \cap (r, 1]| = n + 1\}.
    $$
    It is easy to see that 
    $$
      \mathcal{F}_{n + 1} = \bigcup_{\begin{subarray}{c}
        p, q, r \in \Q \cap (0, 1) \\
        p < q < r 
      \end{subarray}} \mathcal{F}_{n + 1}^{p, q, r}, 
    $$
    hence for some $0 < p, q, r < 1$, $\mu(g\mathcal{F}_{n + 1}^{p, q, r}) > 
    0$. A function $f \in \mathcal{F}_{n + 1}^{p, q, r}$ has no fixed points in 
    $[q, r]$, hence either $f(x) < x$ or $f(x) > x$ for every $x \in [q, r]$. 
    Let ${\mathcal{F}_{n + 1}^{p, q, r}}'$ contain those $f \in 
    \mathcal{F}_{n+1}^{p, q, r}$ with $f(x) > x$ for every $x \in [q, r]$, and 
    let ${\mathcal{F}_{n + 1}^{p, q, r}}''$ contain those with $f(x) < x$ for 
    every $x \in [q, r]$. Then either $\mu(g{\mathcal{F}_{n + 1}^{p, q, r}}') > 
    0$ or $\mu(g{\mathcal{F}_{n + 1}^{p, q, r}}'') > 0$. The rest of the proof 
    of the lemma is similar in these two cases, we only show the proof for the 
    first case. 
    
    So suppose that $\mu(g{\mathcal{F}_{n + 1}^{p, q, r}}') > 0$. This implies 
    $\mu(g{\mathcal{F}_{n + 1}^{p, q, r}}' \cap \mathcal{K}) > 0$ where 
    $\mathcal{K}$ is the support of $\mu$. Now we define a function $h \in 
    \hp([0, 1])$ in the following way: let $h(x) = x$ for $x \in [r, 1]$, let 
    $h(\frac{q + r}{2}) = \frac{q + r}{4}$ and let $h$ be linear on the 
    interval $[\frac{q + r}{2}, r]$. Now $h(\frac{q + r}{2}) < f(\frac{q + 
    r}{2})$ for every $f$ in the closure of ${\mathcal{F}_{n + 1}^{p, q, r}}' 
    \cap g^{-1} \mathcal{K}$. The closure is a compact set, hence we can use 
    Lemma \ref{l:extension of g} to extend $h$ to the interval $[0, \frac{q + 
    r}{2}]$ to be in $\hp([0, 1])$ with the property that $h(x) < f(x)$ for 
    every $x \in (0, \frac{q + r}{2})$ and every $f$ in the closure of 
    ${\mathcal{F}_{n + 1}^{p, q, r}}' \cap g^{-1} \mathcal{K}$. 
    
    One can easily see that for every $f \in h^{-1}({\mathcal{F}_{n + 1}^{p, q, 
    r}}' \cap g^{-1} \mathcal{K})$, $f \in \mathcal{F}_n$, since $h$ was 
    constructed so that the translate of ${\mathcal{F}_{n + 1}^{p, q, r}}' \cap 
    g^{-1} \mathcal{K}$ by $h^{-1}$ preserves every fixed point in $[r, 1]$ but 
    the fixed point in $(p, q)$ vanishes. Consequently, 
    $$
      g{\mathcal{F}_{n + 1}^{p, q, r}}' \cap \mathcal{K} = 
      g({\mathcal{F}_{n + 1}^{p, q, r}}' \cap g^{-1} \mathcal{K}) = 
      ghh^{-1}({\mathcal{F}_{n + 1}^{p, q, r}}' \cap g^{-1} \mathcal{K}) 
      \subset gh\mathcal{F}_n, 
    $$
    hence $\mu(gh\mathcal{F}_n) \ge \mu(g{\mathcal{F}_{n + 1}^{p, q, r}}' \cap 
    \mathcal{K}) > 0$, contradicting our assumption that $\mu$ is a witness 
    measure for $\mathcal{F}_n$. This finishes the proof of the lemma.
  \end{proof}
  The moreover part of Theorem \ref{t:homeo [0,1] conjugacy classes} follows 
  from the first part of the theorem and from the 
  fact that the set $\{f \in \hp([0, 1]) : \fix(f)$ has no limit point in $(0, 
  1)$ and $f(x) - x$ has no local extremum point at any $x_0 \in \fix(f) \cap 
  (0, 1)\}$ is co-Haar null using Claim \ref{c:no limit of fixed points in 
  (0,1)} and Claim \ref{c:no local extremum fixed point is co-Haar null}. 
  Therefore, the proof of the theorem is also complete. 
\end{proof}

\section{Homeomorphisms of the circle}
\label{s:homeo S^1}

In this section we investigate $\hp(\mathbb{S}^1)$, the group of orientation 
preserving homeomorphisms of the circle $\mathbb{S}^1 = \R / \Z$, with the uniform topology. 
The purpose of this section is to prove our main theorem about the group 
$\hp(\mathbb{S}^1)$, Theorem \ref{t:haar positive conj classes in homeo s1}. Before 
turning to the proof, we first introduce the basic definitions and state some 
classical results about circle homeomorphisms. All of these 
results can be found in \cite{KH}. 

For $f \in \hp(\mathbb{S}^1)$, an increasing homeomorphism $F \in 
\hp(\R)$ is a \emph{lift} of $f$, if $f(x) \equiv F(x) \pmod{1}$ for each $x \in [0, 
1)$ and $F(x + 1) = F(x) + 1$ for each $x \in \R$. It is easy to see that if 
$F_1$ and $F_2$ are two lifts of $f \in \hp(\mathbb{S}^1)$ then there exists a $k \in 
\Z$ such that $F_1(x) - F_2(x) = k$ for each $x \in \R$. For $\alpha \in \R$ 
let $R_\alpha \in \hp(\mathbb{S}^1)$ denote the function $x \mapsto x + \alpha 
\pmod{1}$. Then let $\overline{R}_\alpha \in \hp(\R)$ be the homeomorphism 
$\overline{R}_\alpha(x) = x + \alpha$. Note that $R_0 \equiv R_1$, although 
$\overline{R}_0 \not \equiv \overline{R}_1$. 

If $f \in \hp(\mathbb{S}^1)$ and $F$ is a lift of $f$ then the quantity
$$
  \tau(F) = \lim_{n \to \infty} \frac{1}{n} (F^n(x) - x)
$$
is independent of $x$, and is well defined up to an integer, that is, 
$\tau(F) - \tau(F') \in \Z$ if $F'$ is another lift of $f$. Then it 
makes sense to define the \emph{rotation number} of $f$ as $\tau(f) \equiv 
\tau(F) \pmod{1}$, hence $\tau(f) \in \mathbb{S}^1$. It is well-known that $\tau$ is 
continuous and conjugacy invariant. We will use it to classify the 
conjugacy classes in $\hp(\mathbb{S}^1)$ and also to prove some properties of them.

An \emph{orbit} of a point $x \in \mathbb{S}^1$ under $f \in \hp(\mathbb{S}^1)$ is $\{f^n(x) : n 
\in \Z\}$. A point $x \in \mathbb{S}^1$ is a \emph{periodic point} of $f \in \hp(\mathbb{S}^1)$ 
if $f^n(x) = x$ for some $n \in \N$, $n \ge 1$. An orbit is \emph{periodic} if 
it is the orbit of a periodic point. The \emph{period} of a periodic point $x$ 
is the least $n \ge 1$ such that $f^n(x) = x$. It is a well known fact that 
$\tau(f) \in \Q$ if and only if $f$ has periodic points, and if $\tau(f) \in 
\Q$ then all periodic points have the same period. Moreover, if $\tau(f) = p/q$ 
with $p$, $q$ relatively prime, $0 \le p < q$ then every periodic point of 
$f$ has period $q$. 

If $x \in \mathbb{S}^1$ is a fixed point of $f \in \hp(\mathbb{S}^1)$ and $F$ is a lift of $f$ 
then $F(x') = x' + k$ for some $k \in \Z$, where $x' \in \R$ is any real 
number with $x \equiv x' \pmod{1}$. We say that $x$ is a \emph{crossing fixed 
point} if for every neighbourhood $U \subset \R$ of $x'$ there is a point $y_1 
\in U$ with $F(y_1) < y_1 + k$, and there is another point $y_2 \in U$ with  
$F(y_2) > y_2 + k$. If $x$ is a periodic point of $f$ and $n$ is the 
smallest positive integer with $f^n(x) = x$ then we say that $x$ is a 
\emph{crossing periodic point} of $f$ if $x$ is a crossing fixed point of 
$f^n$. 

We introduce the relation $x < y < z$ for three elements $x, y, z \in 
\mathbb{S}^1$, by saying that $x < y < z$ holds if and only if $x' < y' < z'$ or $z' 
< x' < y'$ or $y' < z' < x'$, where $x'$, $y'$ and $z'$ are the unique 
representatives of $x$, $y$ and $z$, respectively, in $[0, 1)$. We also use 
the notation $x_1 < x_2 < \dots < x_n$ for $x_1, \dots, x_n \in \mathbb{S}^1$ if and 
only if $x_i < x_j < x_k$ for $i < j < k$. Also, for $x, y \in \mathbb{S}^1$ let 
$(x, y) = \{z \in \mathbb{S}^1 : x < z < y\}$, and analogously for $[x, y]$, $[x, 
y)$ and $(x, y]$. 

A fixed point $x \in \mathbb{S}^1$ of $f$ is \emph{attractive} if there is an interval 
$(u, v)$ such that $x \in (u, v)$ and for every $y \in (u, v)$, $f^n(y) \to x$. 
A fixed point $x$ is \emph{repulsive} if there is an interval $(u, v)\ni x$  such 
that for every $y \in (u, v) \setminus \{x\}$ and for every large enough $n \in 
\N$, $f^n(y) \not \in (u, v)$. The proofs of the following lemmas are standard 
and are left to the reader. 

\begin{lemma}
  \label{l:attractive, repulsive fixed points}
  Suppose that $f \in \hp(\mathbb{S}^1)$ has some, but only finitely many fixed 
  points. Let $x$ be a fixed point of $f$, and let $y$ and $z$ be two fixed 
  points (not necessarily different from $x$) such that $y < x < z$ and there 
  are no fixed points in the intervals $(y, x)$ and $(x, z)$. Then $x$ is 
  attractive if and only if $f(p) \in (p, x)$ for every $p \in (y, x)$ and 
  $f(p) \in (x, p)$ for every $p \in (x, z)$. Similarly, $x$ is repulsive if 
  and only if $f(p) \not \in (p, x)$ for every $p \in (y, x)$ and $f(p) \not 
  \in (x, p)$ for every $p \in (x, z)$. 
\end{lemma}
\begin{lemma}
  \label{l:crossing fixed points}
  Let $f \in \hp(\mathbb{S}^1)$ be a homeomorphism with finitely many fixed points, all 
  of whom are crossing. Then every fixed point of $f$ is either repulsive or 
  attractive, and the two types are alternating. 
\end{lemma}

Now we describe the conjugacy classes of $\hp(\mathbb{S}^1)$ with rational rotation 
number using results from \cite{GFS}. Suppose that $f \in \hp(\mathbb{S}^1)$ has a fixed 
point. Then for some $x \in \R$ and $k \in \Z$, $F(x) = x + k$, where $F$ is a 
lift of $f$. The \emph{signature} of $f$ is then defined as $\Delta_f(x) = 
\sign(F(x) - x - k)$. Note that value $F(x) - x - k$ only 
depends on the fractional part of $x$, hence it is legitimate to interpret 
$\Delta_f$ as a function from $\mathbb{S}^1$ to $\{-1, 0, 1\}$, which we will do. Also 
notice that if $F(x') = x' + k'$ for another $x' \in \R$ and $k' \in 
\Z$ then $k = k'$, hence $\Delta_f$ is well-defined. Now we state the facts we 
need about conjugation.

\begin{lemma}[\cite{GFS}]
  \label{l:conjugacy in S1}
  If the two homeomorphisms $f, g\in \hp(\mathbb{S}^1)$ have fixed points then they are 
  conjugate if and only if there exists a homeomorphism $h \in \hp(\mathbb{S}^1)$ such 
  that $\Delta_g = \Delta_f \circ h$. If $\tau(f) = \tau(g) = \frac{p}{q}$ with 
  $p$ and $q$ relatively prime and $0 \le p < q$ then $f$ and $g$ are conjugate 
  if and only if $f^q$ and $g^q$ are conjugate. 
\end{lemma}

Now we prove a proposition about the random homeomorphism. 

\begin{proposition}
  \label{p:co-Haar null hom(s1) has finitely many crossing per points}
  The set of $f \in \hp(\mathbb{S}^1)$ with infinitely many periodic points and the set of $f \in \hp(\mathbb{S}^1)$ with finitely many periodic points at least one of which is non-crossing are both Haar null. 
\end{proposition}
\begin{proof}
  Since all the periodic points have the same period, if $f$ has infinitely 
  many periodic points then for some $n$, $f^n$ has infinitely many fixed 
  points. Hence, it is enough to prove that for each $n$, for co-Haar null 
  many $f \in \hp(\mathbb{S}^1)$, $f^n$ only has finitely many fixed points and each of 
  them are crossing. Let $n$ be a fixed positive integer and let $\mathcal{F}_n 
  = \{f \in \hp(\mathbb{S}^1) : f^n$ has finitely many fixed points and all of them are 
  crossing$\}$. It is enough to show that $\mathcal{F}_n$ is co-Haar null. 
  
  As the number of fixed points is the same for conjugate homeomorphisms, and 
  $(h^{-1}fh)^n = h^{-1}f^n h$, the number of periodic points with period $n$ 
  is also the same for conjugate automorphisms. The fact that a homeomorphism 
  has a non-crossing periodic point also holds simultaneously for conjugate 
  homeomorphisms. It follows that $\mathcal{F}_n$ is conjugacy invariant. 
  
  We now claim that $\mathcal{F}_n$ is a Borel set. Let $\mathcal{F} = \{f \in 
  \hp(\mathbb{S}^1) :$ $f$ has finitely many fixed points and all of them are 
  crossing$\}$. It is enough to show that $\mathcal{F}$ is a Borel set, since 
  $\mathcal{F}_n$ is the inverse image of $\mathcal{F}$ under the continuous 
  map $f \mapsto f^n$. 
  
  Let 
  \begin{equation*}
    \mathcal{F}_{k, l} = \{ f \in \hp(\mathbb{S}^1) :  \exists x_1, 
      \dots x_l \in \mathbb{S}^1 \cap \fix(f) 
      \left(i \neq j \Rightarrow |x_i - x_j| \ge 1/k\right) \},
  \end{equation*}
  where the distance $|x - y|$ of two elements $x, y \in \mathbb{S}^1$ are interpreted 
  as $\min(y' - x', x' + 1 - y')$, where $x'$, $y' \in \R$ with $x \equiv x' 
  \pmod{1}$, $y \equiv y' \pmod{1}$ and $x' \le y' < x' + 1$. Also let
  \begin{equation*}
  \begin{split}
  \mathcal{C}_m = \Big\{&f \in \hp(\mathbb{S}^1) : \text{for a lift $F$ of $f$ there 
  exists $x \in [0, 1]$  and $k \in \Z$} \\ &\text{with $F(x) = x + k$ and 
    either $F(y) \le y + k$ for all $y \in \Big(x - \frac{1}{m}, x + 
    \frac{1}{m}\Big)$} \\ 
  &\text{or $F(y) \ge y + k$ 
    for all $y \in \Big(x - \frac{1}{m}, x + \frac{1}{m}\Big)$}\Big\}.
  \end{split}
  \end{equation*}
  Clearly $f \not \in \mathcal{F}$ if and only if $f$ has infinitely many fixed 
  points or it has a non-crossing fixed point, thus
  $$
    \mathcal{F}^c = \bigcap_{l \in \N} \bigcup_{k \in \N} \mathcal{F}_{k, l} 
      \cup \bigcup_{m \in \N} \mathcal{C}_m.
  $$
  Hence, in order to show that $\mathcal{F}$ is Borel, it is enough to prove 
  that $\mathcal{F}_{k, l}$ and $\mathcal{C}_m$ are Borel sets for every $k, l, 
  m \in \N$. 
  
  First we show that $\mathcal{F}_{k,l}$ is closed for every $k, l \in \N$. Let 
  $f_n \in \mathcal{F}_{k, l}$ be a sequence converging uniformly to $f$, 
  and $x^n_i \in \mathbb{S}^1 \cap \fix(f_n)$ for $i = 1, \dots, l$ 
  satisfy $|x^n_i - x^n_j| \ge 1/k$ for $i \neq j$. Then, for an 
  appropriate subsequence, the numbers $x^n_i$ converge to some $x_i$ for all 
  $i$. It is easy to see that $x_i \in \fix(f)$ and $|x_i - x_j| \ge 1/k$ for 
  $i \neq j$. Hence $f \in \mathcal{F}_{k, l}$, showing that 
  $\mathcal{F}_{k, l}$ is closed, thus Borel. 
  
  Now we prove that each $\mathcal{C}_m$ is Borel by showing that it is also 
  closed. Let $(f_k) \subset \mathcal{C}_m$ be a uniformly convergent sequence with 
  $f_k \to f$. For each $k$, let $x_k \in [0, 1]$ and $F_k$ a suitable lift of 
  $f_k$ with $F_k(x_k) = x_k$ and either $F_k(y) \le y$ for every $y \in (x_k - 
  1/m, x_k + 1/m)$, or $F_k(y) \ge y$ for every $y \in (x_k - 1/m, x_k + 1/m)$. 
  Then there exists a subsequence $k_l$ such that $x_{k_l}$ converges to some 
  $x \in [0, 1]$ and we also suppose that for each $l$, $F_{k_l}(y) \le y$ for 
  every $y \in (x_{k_l} - 1/m, x_{k_l} + 1/m)$. The case where for each $l$, 
  $F_{k_l}(y) \ge y$ for every $y \in (x_{k_l} - 1/m, x_{k_l} + 1/m)$, is 
  analogous. Now $x$ is clearly a fixed point of $f$. Hence, for a suitable 
  lift $F$ of $f$, $F(x) = x$, and it is also easy to see that for every $y \in 
  (x - 1/m, x + 1/m)$, $F(y) \le y$, thus $f \in \mathcal{C}_m$. This shows 
  that $\mathcal{C}_m$, and thus $\mathcal{F}_n$ is a Borel set. 
  
  Now we prove that $\mathcal{F}_n$ is co-Haar null by constructing a Borel 
  probability measure $\mu$ on $\hp(\mathbb{S}^1)$ such that for every $f \in \hp(\mathbb{S}^1)$, 
  $\mu(\mathcal{F}_nf^{-1}) = 1$. Then Lemma \ref{l:haar null conjugacy} 
  implies that $\mathcal{F}_n$ is co-Haar null, since $\mathcal{F}_n$ is a 
  conjugacy invariant Borel set. 
  
  Let 
  \begin{equation}
    \label{e:mu def for S1}
    \mu(B) = \lambda(\{\alpha \in [0, 1] : R_\alpha \in B\}),
  \end{equation}
  where $\lambda$ is the Lebesgue measure. It is easy to see that $\mu$ is a 
  Borel probability measure. Let $f \in \hp(\mathbb{S}^1)$ be arbitrary, we need to show 
  that $\mu(\mathcal{F}_nf^{-1}) = 1$, or equivalently, $\lambda(\{\alpha \in 
  [0, 1] : R_\alpha f \in \mathcal{F}_n\}) = 1$.
  
  Let $F$ be the lift of $f$ with $F(0) \in [0, 1)$ and $H = \{(x, \alpha) \in 
  \R \times \R : (\overline{R}_\alpha F)^n(x) - x \in \Z\}$. First of all, it 
  is easy to see that $H$ is closed, since the map $\phi(x, \alpha) = 
  (\overline{R}_\alpha F)^n(x) - x$ is continuous, and $H = 
  \phi^{-1}(\Z)$. Let $k \in \Z$ be fixed. The map $\alpha \mapsto 
  (\overline{R}_\alpha F)^n(x) - x$ is strictly increasing and 
  continuous, and has limits $-\infty$ at $-\infty$ and $\infty$ at $\infty$. 
  Hence, $\phi^{-1}(k)$ is the graph of the function assigning to $x \in \R$ 
  the unique $\alpha \in \R$ with $(\overline{R}_\alpha F)^n(x) - x 
  = k$. Let $\psi_k : \R \to \R$ denote the corresponding function. It is easy 
  to see that $\phi(x + 1, \alpha) = \phi(x, \alpha)$, hence $\psi_k$ is 
  periodic for each $k$ with period 1. Since $F(0) \in [0, 1)$, clearly 
  $F(x) \in [x - 1, x + 1)$ for every $x \in \R$. It follows that 
  for $x \in [0, 1)$, $(\overline{R}_\alpha F)^n(x) - x \in [n 
  \alpha - n - 1, n\alpha + n + 1]$. Hence, $k \in [n \psi_k(x) - n - 1, 
  n\psi_k(x) + n + 1]$, meaning that 
  \begin{equation}
    \label{e:psi(x) is in some interval}
    \psi_k(x) \in [k/n - 1 - 1/n, k/n + 1 + 1/n] \text{ for $x \in [0, 1)$}
  \end{equation}
  and actually for every $x \in \R$ using the periodicity of $\psi_k$. 
  
  Thus $\psi_k$ is bounded and its graph is closed, hence it is continuous. It 
  also satisfies
  \begin{equation}
    \label{e:psi(y) - psi(x) <= valami}
    \psi_k(y) - \psi_k(x) \le y - x \text{ if $x < y$,} 
  \end{equation} 
  since either $\psi_k(y) \le \psi_k(x)$, and it is automatic, or $\psi_k(y) > 
  \psi_k(x)$, in which case it follows from 
  \begin{equation*}
  \begin{split}
    k &=  (\overline{R}_{\psi_k(x)} F)^n(x) - x \le 
    (\overline{R}_{\psi_k(x)} F)^n(y) - y + y - x \\ &\le 
    \psi_k(x) - \psi_k(y) + (\overline{R}_{\psi_k(y)} F)^n(y) 
    - y + y - x \\
    &= \psi_k(x) - \psi_k(y) + k + y - x.
  \end{split}
  \end{equation*}
  Then \eqref{e:psi(y) - psi(x) <= valami} implies that $\psi_k|_{[0, 1)}$ is 
  of bounded variation, since $\psi_k(0) = \psi_k(1)$, hence in any subdivision 
  $0 = a_0 < a_1 < \dots < a_l = 1$, the sums 
  $$
    \sum_{\begin{subarray}c i \\ \psi_k(a_i) < \psi_k(a_{i + 1})\end{subarray}} 
    \psi_k(a_{i + 1}) - \psi_k(a_{i}) \text{ and } 
    \sum_{\begin{subarray}c i \\ \psi_k(a_i) > \psi_k(a_{i + 1})\end{subarray}} 
    \psi_k(a_{i + 1}) - \psi_k(a_{i})
  $$ 
  have the same absolute value, and the former is at most $1$. Using the 
  fact that ${\psi_k|_{[0, 1)}}$ is of bounded variation, a result of  Banach 
  \cite{B} then implies that 
  \begin{equation}
    \label{e:psi_k^{-1}(a) finite for ae a}
    \text{for almost all $\alpha \in [0, 1]$, ${\psi_k|_{[0, 1)}}^{-1}(\alpha)$ 
    is finite for every $k \in \Z$.}
  \end{equation}
  
  To finish the proof of the proposition, we need to show that 
  $\lambda(\{\alpha \in [0, 1] : R_\alpha f \not \in \mathcal{F}_n\}) = 0$. Now 
  we have $\{\alpha \in [0, 1] : R_\alpha f \not \in \mathcal{F}_n\} = \{\alpha 
  \in [0, 1] : ({R}_\alpha {f})^n(x) = x$ for infinitely many $x \in \mathbb{S}^1\} \cup 
  \{\alpha \in [0, 1] : (R_\alpha f)^n$ has finitely many fixed points and at 
  least one of them is non-crossing$\}$. Now we show that both sets are of 
  measure zero, starting with the former. 
  
  It is easy to see that $\{\alpha \in [0, 1] : ({R}_\alpha {f})^n(x) = x$ for 
  infinitely many $x \in \mathbb{S}^1\} = \{\alpha \in [0, 1] : (\overline{R}_\alpha 
  {F})^n(x) - x \in \Z$ for infinitely many $x \in [0, 1)\} = \{\alpha \in [0, 
  1] : \bigcup_k {\psi_k|_{[0, 1)}}^{-1}(\alpha)$ is infinite$\}$. Using 
  \eqref{e:psi(x) is in some interval}, $\psi_k$ has values in $[0, 1]$ for 
  only finitely many $k$, hence the union can be changed to a finite union, 
  meaning that the set above equals $\{\alpha \in [0, 1] : {\psi_k|_{[0, 
  1)}}^{-1}(\alpha)$ is infinite for some $k\}$. And this is clearly of measure 
  zero using \eqref{e:psi_k^{-1}(a) finite for ae a}.
  
  Now we turn to the latter set. Suppose that for some $\alpha \in [0, 1]$, 
  $(R_\alpha f)^n \in \hp(\mathbb{S}^1)$ has finitely many fixed points and one of them 
  is non-crossing. Let $x \in [0, 1]$ represent the non-crossing fixed point, 
  then $(\overline{R}_\alpha F)^n(x) = x + k$ for some $k \in \Z$. 
  Since $(R_\alpha f)^n$ only has finitely many fixed points and $x$ is 
  non-crossing, there is a neighbourhood $(x - \delta, x + \delta)$ of $x$ such 
  that either $(\overline{R}_\alpha F)^n(y) > y + k$ for every $y 
  \in (x - \delta, x + \delta) \setminus \{x\}$ or $(\overline{R}_\alpha 
  F)^n(y) < y + k$ for every $y \in (x - \delta, x + \delta) 
  \setminus \{x\}$. In the first case $x$ is a strict local maximum point of 
  $\psi_k$, in the second case $x$ is a strict local minimum point of $\psi_k$. 
  Hence, we get that $\{\alpha \in [0, 1] : (R_\alpha f)^n$ has finitely many 
  fixed point and at least one of them is non-crossing$\} \subset \{\alpha \in 
  [0, 1] : \alpha$ is a strict local minimum or maximum of $\psi_k$ for some $k 
  \in \Z\}$. And the latter is well known to be countable (see 
  e.g.~\cite[Theorem 7.2]{ST}), thus it is of 
  measure zero. This completes the proof of the proposition. 
\end{proof}

Now we are ready to prove the main theorem of this section. 

\begin{theorem}\label{t:haar positive conj classes in homeo s1} 
  The conjugacy class of $f \in \hp(\mathbb{S}^1)$ is non-Haar null if and only if 
  $\tau(f) \in \Q$, $f$ has finitely many periodic points and all of them are 
  crossing. Every such conjugacy class necessarily contains an even number of 
  periodic orbits, and for every rational number $0 \le r < 1$ and 
  positive integer $k$ there is a unique non-Haar null conjugacy class with 
  rotation number $r$ containing $2k$ periodic orbits. Moreover, the union of 
  the non-Haar null conjugacy classes and $\{f \in \hp(\mathbb{S}^1) : \tau(f) \not \in 
  \Q\}$ is co-Haar null. 
\end{theorem}
\begin{proof}
  For each rational number $r$ let 
  \begin{equation}
  \begin{split}
  \mathcal{F}_r = \{f \in \hp(\mathbb{S}^1) : \; &\tau(f) = r, \text{ $f$ has finitely many periodic points} \\ &\text{and each of them is crossing}\}.
  \end{split}
  \end{equation}
  It is clear that each set $\mathcal{F}_r$ is the union of conjugacy 
  classes, since the rotation number is conjugacy invariant. 

  We only prove the first assertion of the theorem. The second will follow 
  using Claim \ref{c:f conj g iff number of periodic points is the same} and 
  Claim \ref{c:s1 homeo exists iff num of orbits is even} below. The moreover 
  part follows from Proposition \ref{p:co-Haar null hom(s1) has finitely many 
  crossing per points}. 
  
  Let us prove the ``only if'' part of the first statement. Using Proposition 
  \ref{p:co-Haar null hom(s1) has finitely many crossing per points}, it is 
  enough to show that if $\tau(f) \not \in \Q$ then the conjugacy class of $f$ 
  is Haar null. Let $\mathcal{C} = \{g \in \hp(\mathbb{S}^1) : \tau(g) = \tau(f)\}$. It 
  is enough to show that $\mathcal{C}$ is Haar null. In order to do so, we use 
  the measure $\mu$ defined in \eqref{e:mu def for S1}. Note that $\mathcal{C}$ 
  is closed, since the map $f \mapsto \tau(f)$ is continuous. Hence, using 
  Lemma \ref{l:haar null conjugacy} and the fact that $\mathcal{C}$ is 
  conjugacy invariant, it is enough to show that $\mu(\mathcal{C} g) = 0$ for 
  every $g \in \hp(\mathbb{S}^1)$. Clearly, $\{\alpha \in [0, 1] : R_\alpha \in 
  \mathcal{C}g\} = \{\alpha \in [0, 1] : \tau(R_\alpha g^{-1}) = \tau(f)\}$.  
  Let $G$ be a lift of $g^{-1}$. Using \cite[Proposition 11.1.9]{KH}, the map 
  $\alpha \mapsto \tau(\overline{R}_\alpha G)$ is strictly increasing at points 
  where $\tau(\overline{R}_\alpha G) \not \in \Q$. Hence the inverse image of 
  $\Z + \tau(f)$ is countable, thus $\{\alpha \in [0, 1] : \tau(R_\alpha 
  g^{-1}) = \tau(f)\}$ is countable, finishing the proof that $\mathcal{C}$ is 
  indeed Haar null. 
  
  Now we prove the ``if'' part. First we describe the conjugacy classes that 
  satisfy the conditions of the theorem. 

  \begin{claim}
    \label{c:f conj g iff number of periodic points is the same}
    Let $f, g \in \mathcal{F}_r$ for some $r \in \Q$. Then $f$ and $g$ are 
    conjugate if and only if they have the same number of periodic points.
  \end{claim}
  \begin{proof}
    Let $r = p/q$ with $p$, $q$ relatively prime, $q > 0$. If $f$ and $g$ are 
    conjugate then $f^q$ and $g^q$ are also conjugate, hence $f^q$ and $g^q$ 
    have the same number of fixed points, thus indeed, $f$ and $g$ have the 
    same number of periodic points. Now suppose that they have the same number 
    of periodic points. Then $f^q$ and $g^q$ have the same number of fixed 
    points, moreover, $f^q, g^q \in \mathcal{F}_0$. Using the first statement 
    of Lemma \ref{l:conjugacy in S1}, it is easy to see that if two 
    homeomorphisms from $\mathcal{F}_0$ have the same number of fixed points 
    then they are conjugate. Hence $f^q$ and $g^q$ are conjugate, but then, 
    using Lemma \ref{l:conjugacy in S1} again, $f$ and $g$ are also conjugate. 
    This proves the claim.
  \end{proof}
  
  \begin{claim}
    \label{c:s1 homeo exists iff num of orbits is even}
    Let $p$, $q$ be relatively prime numbers with $0 \le p < q$ and let 
    $n > 0$. We claim that there exists a homeomorphism $f \in 
    \mathcal{F}_{p/q}$ with exactly $n$ periodic points if and only if the 
    number of periodic orbits, $n / q$ is even.
  \end{claim}
  \begin{proof}
    Let $p, q$ and $n$ be as above with $n / q$ even, we first show that such 
    homeomorphisms exist. Let $F$, an increasing homeomorphism of $\R$ be 
    defined as follows. For every $k \in \{0, 1, \dots, n - 1\}$, $F$ maps the 
    interval $[\frac{k}{n}, \frac{k + 1}{n}]$ bijectively (increasingly and 
    continuously) into $[\frac{k + (np)/q}{n}, \frac{k + 1 + (np)/q}{n}]$. We 
    also require that if $k$ is even then every $y \in (\frac{k}{n}, \frac{k + 
    1}{n})$, is ``shifted up'', that is, $F(y) - \frac{k + (np)/q}{n} > y - 
    \frac{k}{n}$. On the other hand, if $k$ is odd then we require for such 
    points that $F(y) - \frac{k + (np)/q}{n} < y - \frac{k}{n}$. Then we 
    defined $F$ on $[0, 1]$ as a strictly increasing, continuous function with 
    $F(1) = F(0) + 1$, hence we can extend it uniquely to a homeomorphism of 
    $\R$ (which we also denote by $F$) with $F(x + 1) = F(x) + 1$ for every $x 
    \in \R$. Thus, $F$ is a lift of a homeomorphism $f \in \hp(\mathbb{S}^1)$. 
    
    We need to show that $f \in \mathcal{F}_{p/q}$ and it has exactly $n$ 
    periodic points. For every $k, \ell \in \N$, $F^{\ell q}(k/n) = k/n + \ell 
    p$, from which it follows that $\tau(f) = p/q$ and it has at least $n$ 
    periodic points. Let $y \in (\frac{k}{n}, \frac{k + 1}{n})$ for some 
    natural number $k$ with $0 \le k < n$. If $k$ is even then $k + (np)/q$ is 
    also even and so is $k + \ell (np)/q$ for every $\ell \in \N$. Hence 
    $y - 
    \frac{k}{n} < F(y) - \frac{k + (np)/q}{n} < \dots < F^q(y) - \frac{k + 
    (np)}{n} = F^q(y) - \frac{k}{n} + p$. If $k$ is odd then one can similarly 
    obtain $y - \frac{k}{n} > F^q(y) - \frac{k}{n} + p$. These facts imply that 
    in both cases, $y$ is not a fixed point of $f^q$, hence it is not a 
    periodic point of $f$, since all periodic points of $f$ have the same 
    period, $q$. On the other hand, these facts imply that each fixed point of 
    $f^q$ is crossing, showing that $f$ satisfies the required properties. 
    
    Now we prove that if such a homeomorphism $f$ exists then $n/q$ is even. 
    Suppose indirectly that $n/q$ is odd. It is clear from the fact that each 
    fixed point of $f^q$ is crossing that the number of fixed points of $f^q$ 
    is even, hence $n$ is even. Thus $q$ is even and $p$ is odd. Let $0 \le x_0 
    < x_1 < \dots < x_{n - 1} < 1$ represent the periodic points of $f$. Since 
    they form orbits of size $q$, there is a $k < q$ such that for every $i 
    < n$, $f(x_i) = x_{j(i)}$, where $j(i) \equiv i + \frac{n}{q}k \pmod{n}$. 
    Here $k$ and $q$ are relatively prime, since otherwise the size of the 
    orbits would be less then $q$, hence $k$ is odd. Actually, $k = p$, but we 
    do not need this fact. The fixed points of $f^q$ are exactly the points 
    $x_i$, hence it makes sense to define the map $\phi: \{0, 1, \dots, n - 1\} 
    \to \{-1, 1\}$ so that $\phi(i) = 1$ if and only if the values over the 
    interval $(x_i, x_{i + 1})$ are above the diagonal with respect to the 
    unique lift of $f^q$ where the fixed points of $f^q$ are on the diagonal. 
    Since every fixed point of $f^q$ is crossing, the values of $\phi$ are 
    alternating. 
    
    Let us assume that $\phi(0) = 1$, the case where $\phi(0) = -1$ can be 
    handled similarly. This means that if we choose an arbitrary point $y \in 
    (x_0, x_1)$ then $(f^q)^\ell(y)$ converges to $x_1$ as $\ell$ tends to 
    infinity, hence $f((f^q)^\ell(y))$ converges to $x_{j(1)}$. Since $k$ and 
    $n/q$ are odd numbers, $j(0)$ is odd, hence $\phi(j(0)) = -1$. 
    Consequently, for any point in $(x_{j(0)}, x_{j(0) + 1})$, for example, for 
    $f(y)$, we have that $(f^q)^\ell(f(y)) \to x_{j(0)}$ as $\ell \to \infty$. 
    But $(f^q)^\ell(f(y)) = f((f^q)^\ell(y))$, a contradiction. Hence the proof 
    of the claim is complete. 
  \end{proof}
  
  Now let $p, q \in \N$ be relatively prime numbers with $0 \le p < q$, $k \in 
  \N$ with $k \ge 1$ and let $\mathcal{F}_{p/q}^k = \{f \in \mathcal{F}_{p/q} : 
  |\fix(f^q)| = 2kq\}$. Using Claim \ref{c:s1 homeo exists iff num of orbits is 
  even} and Claim \ref{c:f conj g iff number of periodic points is the same}, 
  $\mathcal{F}_{p/q}^k$ is a conjugacy class and $\bigcup_{k \ge 1} 
  \mathcal{F}_{p/q}^k = \mathcal{F}_{p/q}$. Hence to finish the proof of the 
  theorem, we need to show that each $\mathcal{F}_{p/q}^k$ is non-Haar null. We 
  do so by proving the following two facts. First, we show that 
  $\mathcal{F}_{p/q}^k$ is non-Haar null for infinitely many $k$, that is, 
  $\bigcup_{k \ge \ell} \mathcal{F}_{p/q}^k$ is non-Haar null for every $\ell 
  \ge 1$. Then we show that if $\mathcal{F}_{p/q}^k$ is Haar null then 
  $\mathcal{F}_{p/q}^{k + 1}$ is also Haar null. One can easily see that these 
  two facts imply that $\mathcal{F}_{p/q}^k$ is non-Haar null for every $k \ge 
  1$. Now we prove these claims one after the other. 
  
  \begin{claim}
    For every $\ell \ge 1$, $\bigcup_{k \ge \ell} \mathcal{F}_{p/q}^k$ is 
    non-Haar null. 
  \end{claim}
  \begin{proof}
    Let $\ell \ge 1$ be fixed and $f \in \mathcal{F}_{p/q}^\ell$. Since every 
    fixed point of $f^q$ is crossing, there is an open neighbourhood 
    $\mathcal{U}$ of $f^q$ such that $|\fix(g)| \ge 2\ell q$ for every $g \in 
    \mathcal{U}$. Let $\mathcal{U}'$ be the inverse image of $\mathcal{U}$ 
    under the continuous map $g \mapsto g^q$. Now $f \in \mathcal{U}'$ and for 
    every $g \in \mathcal{U}'$, $|\fix(g^q)| \ge 2\ell q$. Since the map $g 
    \mapsto \tau(g)$ from $\hp(\mathbb{S}^1)$ to $\mathbb{S}^1$ is continuous (see e.g.~ 
    \cite[11.1.6]{KH}), we can also suppose by shrinking $\mathcal{U}'$ that if 
    $\tau(g) = r/q$ for some $0 \le r < q$ and $g \in \mathcal{U}$ then $r = 
    p$. This observation and the fact that $\fix(g^q) \neq \emptyset$ for 
    $\mathcal{U}'$ imply that $\tau(g) = p/q$ for every $g \in \mathcal{U}'$. 
    
    Since $\mathcal{U}'$ is a non-empty open set, it is non-Haar null. Hence to 
    finish the proof of the claim, it is enough to show that the set 
    $$
      \mathcal{H} = \mathcal{U}' \setminus \bigcup_{k \ge \ell} 
      \mathcal{F}_{p/q}^k
    $$
    is Haar null. But if $f \in \mathcal{H}$ then either $f$ has infinitely 
    many periodic points, or $f$ has a non-crossing periodic point. Hence 
    Proposition \ref{p:co-Haar null hom(s1) has finitely many crossing per 
    points} implies that $\mathcal{H}$ is indeed Haar null, finishing the proof 
    of the claim. 
  \end{proof}
  Now we finish the proof of the theorem by showing the following.
  \begin{claim}
    If $\mathcal{F}_{p/q}^k$ is Haar null for some $k \ge 1$ then 
    $\mathcal{F}_{p/q}^{k + 1}$ is also Haar null.
  \end{claim}
  \begin{proof}
    Suppose indirectly that $\mathcal{F}_{p/q}^k$ is Haar null, but 
    $\mathcal{F}_{p/q}^{k + 1}$ is not. Then, since $\mathcal{F}_{p/q}^{k + 1}$ 
    is conjugacy invariant, there exists a Borel probability 
    measure $\mu$ such that for every $f \in \hp(\mathbb{S}^1)$, 
    $\mu(f\mathcal{F}_{p/q}^k) = 0$, but $\mu(g\mathcal{F}_{p/q}^{k + 1}) > 0$ 
    for some $g \in \hp(\mathbb{S}^1)$. We denote by $K$ the number of periodic points 
    of a homeomorphism $f \in \mathcal{F}_{p/q}^{k + 1}$, that is, $K = q(2k + 
    2)$. 
    
    Let $r_0 < s_0 < r_1 < s_1 < \dots < r_{K - 1} < s_{K - 1}$ 
    be rational elements of $\mathbb{S}^1$, ordered cyclically. For such numbers, let
    $\mathcal{F}_{p/q}^{k + 1}(r_0, s_0, \dots, r_{K - 1}, s_{K - 1}) \subset 
    \mathcal{F}_{p/q}^{k + 1}$ contain $f \in \mathcal{F}_{p/q}^{k + 1}$ if and 
    only if each interval $(r_i, s_i)$ contains exactly one periodic point of 
    $f$. It is easy to see that each $f \in \mathcal{F}_{p/q}^{k + 1}$ is 
    contained in some $\mathcal{F}_{p/q}^{k + 1}(r_0, s_0, \dots, r_{K - 1}, 
    s_{K - 1})$, hence for at least one of these sets, 
    $\mu(g\mathcal{F}_{p/q}^{k + 1}(r_0, s_0, \dots, r_{K - 1}, s_{K - 1})) > 
    0$. We now fix $r_0$, $s_0, \dots$, $r_{K - 1}$, $s_{K - 1}$ in such a way, 
    and let $\mathcal{F}^+ \subset \mathcal{F}_{p/q}^{k + 1}(r_0, s_0, \dots, 
    r_{K - 1}, s_{K - 1})$ contain $f \in \mathcal{F}_{p/q}^{k + 1}(r_0, s_0, 
    \dots, r_{K - 1}, s_{K - 1})$ if and only if the unique fixed point 
    of $f^q$ in the interval $(r_0, s_0)$ is repulsive.
    Similarly, let $\mathcal{F}^- \subset \mathcal{F}_{p/q}^{k + 1}(r_0, s_0, 
    \dots, r_{K - 1}, s_{K - 1})$ contain $f$ if and only if the fixed point 
    of $f^q$ in $(r_0, s_0)$ is attractive. It is easy to see that 
    $\mathcal{F}_{p/q}^{k + 1}(r_0, s_0, \dots, r_{K - 1}, s_{K - 1}) = 
    \mathcal{F}^+ \cup \mathcal{F}^-$, hence we can suppose without loss of 
    generality that $\mu(g\mathcal{F}^+) > 0$: If $\mu(g\mathcal{F}^+) = 0$ 
    then $\mu(g\mathcal{F}^-) > 0$ and by rearranging the intervals cyclically 
    (e.g. $r_{K - 1}' = r_0$, $s_{K - 1}' = s_0$, and $r_i' = r_{i + 1}$, $s_i' = s_{i + 1}$ 
    otherwise) we can achieve that the sets $\mathcal{F}^+$ and $\mathcal{F}^-$ are swapped, using 
    the fact that attractive and repulsive fixed points of $f^q$ are alternating by 
    Lemma \ref{l:crossing fixed points}. 
    
    To get a contradiction, we show that $h\mathcal{F}^+ \subset 
    \mathcal{F}_{p/q}^k$ for some $h \in \hp(\mathbb{S}^1)$, hence $\mathcal{F}^+ 
    \subset h^{-1}\mathcal{F}_{p/q}^k$, thus $\mu(gh^{-1}\mathcal{F}_{p/q}^k) > 
    0$. 
    
    For each $f \in \mathcal{F}^+$, let $p^f_i \in (r_i, s_i)$ be the unique 
    periodic point of $f$ in the interval $(r_i, s_i)$. It follows easily from 
    Lemma \ref{l:attractive, repulsive fixed points} that if $x \in (p_0^f, 
    p_1^f)$ then the sequence $(f^{nq}(x))_{n \in \N}$ is monotone and contained in $(p_0^f, 
    p_1^f)$, thus it converges. The limit is necessarily a fixed point of 
    $f^q$. Since $p_0^f$ is repulsive and hence $p_1^f$ is attractive, we have the 
    following: 
    \begin{equation}
      \label{e:p_1 is attractive}
      \forall f \in \mathcal{F}^+ \; \forall x \in (p_0^f, p_1^f) \left(
      f^{nq}(x) \to p_1^f\right).
    \end{equation}
    
    It is easy to see that $f$ maps a fixed point of $f^q$ to another fixed 
    point in a monotone way, that is, there is an integer $\ell$ with $0 \le 
    \ell < K$ such that $f(p_i^f) = p^f_{j(i)}$, where $j(i) \equiv i + \ell 
    \pmod{K}$ for every $i$. Let us denote by $P_0^f \in [0, 1)$ the real 
    number corresponding to $p_0^f$ and for every $i < K$, let $P_i^f \in \R$ 
    be the real number corresponding to $p_i^f$ with $P_0^f < P_1^f < \dots < 
    P_{K - 1}^f < P_0^f + 1$. For $n \in \N$ we also use the notation $P_n^f = 
    P_r^f + m$, where $n = mK + r$ with $0 \le r < K$. It is easy to see that 
    $P_0^f < P_1^f < \dots P_{K - 1}^f < P_K^f < \dots$. Let $F$ be the lift of 
    $f$ with $F(P_0^f) = P_\ell^f$. It follows that $F^{iK}(P_0^f) = 
    P_{iK\ell}^f = P_0^f + i\ell$ for every $i \in \N$, hence $\tau(F) = 
    \ell/K$. Since $\tau(F) \equiv \tau(f) \pmod{1}$, and $\tau(f) \equiv p/q 
    \pmod{1}$, and both $\ell/K$ and $p/q$ are numbers from $[0, 1)$, it 
    follows that $\ell/K = p/q$. Hence $\ell = p(2k + 2)$, the same for each $f 
    \in \mathcal{F}^+$. To simplify notation, for each $n \in \N$ we denote 
    by $p_n^f, r_n$ and $s_n$ the elements of $\mathbb{S}^1$ equal to $p_{r}^f, r_{r}$ 
    and $s_{r}$, respectively, where $0 \le r < K$ and $n \equiv r \pmod{K}$. 
    
    Now we define a homeomorphism $h \in \hp(\mathbb{S}^1)$ such that the periodic 
    points $p_{i\ell + 1}^f$ and $p_{i\ell + 2}^f$ ($i = 0, 1, \dots, 
    q - 1$) are not periodic for $hf$. Let $\mathcal{P}'_f$ denote the set of these 
    periodic points, and let $\mathcal{P}_f$ denote the set of all periodic points of 
    $f$. Note that $\mathcal{P}'_f$ consists of two periodic orbits.

    For $i = 0, 1, \dots, q - 1$ let the points $x_i, y_i, u_i$ and $v_i$ be 
    chosen so that $s_{i\ell} < x_i < y_i < r_{i\ell + 1}$ and 
    $s_{i\ell + 2} < u_i < v_i < r_{i\ell + 3}$. We use the convention $x_n = 
    x_{r}$, $y_n = y_{r}$, $u_n = u_{r}$ and $v_n = v_{r}$, where $0 \le r < q$ 
    and $n \equiv r \pmod{q}$. Now we define $h$ the following way. On each 
    interval of the form $[x_i, v_i]$, let $h$ map $[x_i, y_i]$ to $[x_i, u_i]$ 
    linearly, and also map $[y_i, v_i]$ to $[u_i, v_i]$ linearly, so that $h$ 
    is order preserving. Then $h(x_i) = x_i$ and $h(v_i) = v_i$, hence it makes 
    sense to require that $h(x) = x$ if $x$ is outside each of these intervals. 
    It is clear that 
    \begin{equation}
      \label{e:h is above the diagonal}
      \forall f \in \mathcal{F}^+ \; \forall n \in \N \; \forall y \in 
      (p_{n\ell}^f, p_{n\ell + 3}^f) \left(h(y) \in [y, p_{n\ell + 3}^f)\right).
    \end{equation}
    We now claim that $hf \in \mathcal{F}_{p/q}^k$ for every $f \in 
    \mathcal{F}^+$. 
    
    Let $f \in \mathcal{F}^+$. If $p\in \mathcal{P}_f \setminus \mathcal{P}'_f$ then 
    it is easy to see that $(hf)^n(p) = f^n(p)$ for every $n \in \N$, hence $p$ is a 
    periodic point of $hf$ with period $q$, hence $hf$ has at least $2kq$ periodic 
    points, that is, at least $2k$ periodic orbits. 
    It is also an immediate consequence that $\tau(hf) = p/q$, hence 
    \begin{equation}
    \label{e:hf is q-periodic}
    \text{every periodic point of $hf$ has period $q$.}
    \end{equation}
    We now show that no other point is periodic. Let 
    \begin{equation}
    \label{e:def of x}
    x \in (\mathbb{S}^1 \setminus \mathcal{P}_f) \cup \mathcal{P}'_f    
    \end{equation}
    be such a point, we complete the proof by showing that $x$ is not periodic. 
    We distinguish several cases depending on $x$. 
    
    Case 1: there is no $n \in \N$ with $x \in (p_{n\ell}^f, p_{n\ell + 3}^f)$. 
    Then the same holds for $f^m(x)$, that is, for every $m \in \N$ there is no 
    $n \in \N$ with $f^m(x) \in (p_{n\ell}^f, p_{n\ell + 3}^f)$. Since $h$ is 
    the identity restricted to the complement of the intervals of the form 
    $(p_{n\ell}^f, p_{n\ell + 3}^f)$, one can show by induction that 
    \begin{equation}
    \label{e:f(x) = hf(x)} 
    f^m(x) = (hf)^m(x) \text{ for every $m \in \N$.}
    \end{equation}
    From the fact that $x \not \in (p_{n\ell}^f, p_{n\ell + 3}^f)$ for every $n 
    \in \N$, it follows that $x \not \in \mathcal{P}'_f$. Using also \eqref{e:def of x}, 
    it is clear that $x$ is not a periodic point of $f$, hence by \eqref{e:f(x) 
    = hf(x)}, it is not a periodic point of $hf$.
    
    Case 2: there exists an $n \in \N$ such that 
    \begin{equation}
      \label{e:x in (p_0, p_3)}
      x \in (p_{n \ell}^f, p_{n \ell + 3}^f).
    \end{equation}
    In order to show that $x$ is not a periodic point of $hf$, using 
    \eqref{e:hf is q-periodic}, it is enough to prove that 
    \begin{equation}
    \label{e:hf^q(x) > x}
    p_{n\ell}^f < x < (hf)^q(x) < p_{n\ell + 3}^f.
    \end{equation}
    We first show that 
    \begin{equation}
    \label{e:hf^q(x) > f^q(x)}
    p_{n\ell}^f < f^q(x) \le (hf)^q(x) < p_{n \ell + 3}^f.
    \end{equation}
    In order to do so, it is enough to prove that 
    \begin{equation}
      \label{e:hf^n(x) > f^n(x)}
      p_{(n + m)\ell}^f < f^m(x) \le (hf)^m(x) < p_{(n + m) \ell + 3}^f
    \end{equation}
    for every $m \in \N$, since applying \eqref{e:hf^n(x) > f^n(x)} with $m = 
    q$ gives \eqref{e:hf^q(x) > f^q(x)}, using that $\ell q = Kp \equiv 0 
    \pmod{K}$. 
    For $m = 0$, \eqref{e:hf^n(x) > f^n(x)} is immediate from \eqref{e:x in 
    (p_0, p_3)}. Now suppose that \eqref{e:hf^n(x) > f^n(x)} is true for some 
    $m \in \N$. Then, using that $f$ is an orientation preserving 
    homeomorphism, $p_{(n + m + 1)\ell}^f < f^{m + 1}(x) \le f((hf)^m(x)) < 
    p_{(n + m + 1) \ell + 3}^f$. Using \eqref{e:h is above the diagonal} with 
    $n = n + m + 1$ and $y = f((hf)^m(x))$, we get \eqref{e:hf^n(x) > f^n(x)} for 
    $m + 1$. Now we distinguish three subcases.
    
    Case 2a: $x \in (p_{n \ell}^f, p_{n \ell + 1}^f)$. Then, using the fact 
    that $p_{n \ell}^f$ is a repulsive fixed point of $f^q$ and Lemma 
    \ref{l:attractive, repulsive fixed points}, $p_{n \ell}^f < x < f^q(x) < 
    p_{n\ell + 1}^f$. This, together with \eqref{e:hf^q(x) > f^q(x)} implies 
    \eqref{e:hf^q(x) > x}. 
    
    Case 2b: $x \in (p_{n \ell + 2}^f, p_{n \ell + 3}^f)$. Then, using Lemma 
    \ref{l:crossing fixed points} the fact that $p_{n \ell}^f$ is a repulsive 
    fixed point of $f^q$, $p_{n \ell + 2}^f$ is also a repulsive fixed point of 
    $f^q$. As in the previous case, Lemma \ref{l:attractive, repulsive fixed 
    points} implies $p_{n \ell + 2}^f < x < f^q(x) < p_{n \ell + 3}^f$. And 
    again using \eqref{e:hf^q(x) > f^q(x)}, this implies \eqref{e:hf^q(x) > x}. 
    
    Case 2c: $x \in [p_{n \ell + 1}^f, p_{n \ell + 2}^f]$. Then $f(x) \in 
    [p_{(n + 1)\ell + 1}^f, p_{(n + 1)\ell + 2}^f]$. Using that $(y_{n + 1}, u_{n + 1}) 
    \supset [p_{(n + 1)\ell + 1}^f, p_{(n + 1)\ell + 2}^f]$ and $(u_{n + 1}, 
    v_{n + 1}) \subset (p_{(n + 1)\ell + 2}^f, p_{(n + 1)\ell + 3}^f)$, and also 
    the fact that $h$ maps $(y_{n + 1}, u_{n + 1})$ to $(u_{n + 1}, v_{n + 
    1})$, we have 
    \begin{equation}
    \label{e:hf(x) jo helyen}
    hf(x) \in (p_{(n + 1)\ell + 2}^f, p_{(n + 1)\ell + 3}^f). 
    \end{equation}
    Applying \eqref{e:hf^n(x) > f^n(x)} with $hf(x)$ in place of $x$ and $n + 
    1$ in place of $n$ and $m = q - 1$, we get 
    \begin{equation}
    \label{e:nagyon vege van mar}
    p_{(n + q)\ell}^f = p_{n\ell}^f < f^{q - 1}(hf(x)) \le (hf)^q(x) < 
    p_{(n + q)\ell + 3}^f = p_{n\ell + 3}^f,
    \end{equation}
    using again that $q \ell = Kp \equiv 0 \pmod{K}$. 
    From \eqref{e:hf(x) jo helyen}, we have $f^{q - 1}(hf(x)) \in (p_{(n + 
    q)\ell + 2}^f, p_{(n + q)\ell + 3}^f) = (p_{n\ell + 2}^f, p_{n\ell + 
    3}^f)$, hence $(hf)^q(x) \in (p_{n\ell + 2}^f, p_{n\ell + 3}^f)$, using 
    \eqref{e:nagyon vege van mar}. Since in Case 2c, $x$ was chosen from $[p_{n 
    \ell + 1}^f, p_{n \ell + 2}^f]$, this implies \eqref{e:hf^q(x) > x} 
    finishing the proof of the claim.
  \end{proof}
  Therefore the proof of the theorem is also complete.
\end{proof}

\section{The unitary group}
\label{s:U(l^2)}
  
  In this section we consider $\mathcal{U}(\ell^2)$, the group of unitary transformations 
  of the Hilbert space $\ell^2$. 
  We will use some well known facts from the spectral theory of unitary 
  operators following the theory developed in \cite[Appendix]{parry} (see also 
  \cite{krit}) up to Theorem \ref{t:ul2conjugate}. Let $U \in \mathcal{U}(\ell^2)$ and $x \in \ell^2$, then the \emph{cyclic 
  subspace generated by $x$} is the closure of the linear span of $\{U^n(x):n 
  \in \mathbb{Z}\}$ and we denote this space by $Z(x)$. If $x$ has norm one then
  by Bochner's theorem there exists a unique probability measure $\mu_x$ on 
  $\mathbb{S}^1$ such that for every $n \in \Z$ \[\langle U^n(x),x\rangle=\int_{\mathbb{S}^1} z^n d\mu_x(z).\]
It can be shown that if $Z(x) \subset Z(y)$ then $\mu_x 
  \ll \mu_y$. If $Z(x)$ is maximal, i.~e.~it is 
  not a proper subspace of any other cyclic subspace, then $\mu_x$ is called a 
  \textit{spectral measure of $U$}. Moreover, if $Z(x)$ is maximal then for 
  every $x'$ we have $\mu_{x'} \ll \mu_x$. Thus, if $x$ and $x'$ generate 
  maximal cyclic subspaces then $\mu_x \simeq \mu_{x'}$ (that is, $\mu_x \ll 
  \mu_{x'}$ and $\mu_{x'} \ll \mu_{x}$) so the spectral 
  measure is unique up to $\simeq$. 
  
  Moreover, we can decompose $\ell^2$ into the orthogonal direct sum of cyclic subspaces as 
  $\ell^2=\oplus_{i \in\N} Z(x_i)$ such that $\mu_{x_0} \gg  \mu_{x_1} \gg 
  \dots$. Let $A_i \subset \mathbb{S}^1$ be the support of the function 
  $\frac{d\mu_{x_i}}{d\mu_{x_1}}$ and define a function $n:\mathbb{S}^1 \to \mathbb{N} 
  \cup \{\infty\}$ as $n=\sum_i \chi_{A_i}$. This function is called a 
  \textit{multiplicity function of $U$} and it is defined $\mu_{x_1}$-almost 
  everywhere. It can be shown that the multiplicity function is also unique a.e. with respect to a spectral measure, i. e. it does not depend on the choice 
  of the vectors $x_i$. Thus, it makes sense 
  to talk about \textit{the} spectral measure $\mu_U$ and \textit{the} multiplicity function 
  $n_U$ of an operator $U$.
  
  From the description of the process we can derive the following simple 
  corollary.
  
  \begin{corollary}
    \label{c:rotate}
    Let $U \in \mathcal{U}(\ell^2)$ be given with spectral measure $\mu_U$ and multiplicity 
    function $n_U$, and let $q \in \mathbb{S}^1$. Then $q U \in  \mathcal{U}(\ell^2)$, $\mu_{q 
    U}=q_*\mu_{U} $ and $n_{q U}=n \circ q^{-1}$ where $q$ and $q^{-1}$ denote 
    the multiplication on $\mathbb{S}^1$ by $q$ and $q^{-1}$, and $q_*\mu_U$ denotes the 
    push-forward measure of $\mu_U$ with respect to the multiplication by $q$. 
  \end{corollary}
    
  The following important theorem states that the spectral measure and the 
  multiplicity function fully characterizes the conjugacy classes in 
  $\mathcal{U}(\ell^2)$.
  
  \begin{theorem}
    \label{t:ul2conjugate} Let $U_1,U_2 \in \mathcal{U}(\ell^2)$ and let $\mu_1,\mu_2$ be 
    spectral measures and $n_1,n_2$ multiplicity functions of $U_1$ 
    and $U_2$. Then $U_1$ and $U_2$ are conjugate iff $\mu_1 \simeq \mu_2$ and 
    $n_1=n_2$ holds $\mu_1$-almost everywhere.  
  \end{theorem}

  Now we are ready to prove our statement about the non-Haar null classes.
  
  \begin{theorem} \label{t:ul2main}
    Let $U \in \mathcal{U}(\ell^2)$ be given with spectral measure $\mu_U$ and multiplicity 
    function $n_U$. If the conjugacy class of $U$ is non-Haar null then $\mu_U 
    \simeq \lambda$ and $n_U$ is constant $\lambda$-a.~e., where $\lambda$ denotes Lebesgue measure.
  \end{theorem}
    
    \begin{proof}
      Suppose that $U$ has a non-Haar null conjugacy class $C^U$. Then by \cite[Corollary 
      5.2]{auto} 
      we have that for a small enough $\varepsilon>0$ for every $q \in \mathbb{S}^1$ if 
      $|q-1|<\varepsilon$ then $qI \in (C^U)^{-1}C^U$. Hence for each 
      $q \in \mathbb{S}^1$ with $|q-1|<\varepsilon$ there exist $U', U'' \in C^U$ with $qU' = U''$. Thus by Corollary 
      \ref{c:rotate} and Theorem \ref{t:ul2conjugate} we have $\mu_{qU} = q_* \mu_{U} \simeq q_*\mu_{U'} = \mu_{qU'} = \mu_{U''} \simeq \mu_{U}$ and similarly $n_{qU}=n_U \circ q^{-1} = n_{U'} \circ q^{-1} = n_{qU'} = n_{U''} = n_U$ holds $\mu_U$-almost everywhere. 
Thus, if $r \in \mathbb{S}^1$ is arbitrary we can choose $k \in \N$ and $q$ with 
      $|q-1|<\varepsilon$ such that $q^k=r$, so pushing forward $\mu_U$ $k$ times by
       multiplication by $q$ we have that $\mu_{rU} = r_*\mu_U \simeq \mu_U$. This means 
      that $\mu_U$ is quasi-invariant probability measure on $\mathbb{S}^1$ therefore 
      $\mu_U \simeq \lambda$ (see e.g. \cite[Theorem 3.1.5]{DXB}). 
      
      Similarly, $n_{rU} = n_U \circ r^{-1}=n_U$ $\mu_U$-almost everywhere. Therefore, $n_U$ is $\mu_U$-almost everywhere constant.
    \end{proof}
    
    Let $k \in \mathbb{N} \cup \{\infty\} \setminus \{0\}$ and let $(e_{i,j})_{i \in k, j \in 
    \mathbb{Z}}$ be an injective enumeration of an orthonormal basis of $\ell^2$ 
    and consider the unitary transformation $U_k$ given by $e_{i,j} \mapsto 
    e_{i,j+1}$. These operators are usually called \emph{multishifts}. It is not hard to 
    see that the spectral measure of such a transformation $\mu_{U_k}$ is 
    equivalent to $\lambda$, while the multiplicity function $n_{U_k}\equiv k$ 
    $\lambda$-almost everywhere. From this and Theorem \ref{t:ul2main} we 
    obtain the following corollary. 
    
    \begin{corollary}\label{c:multishift}
  In the unitary group every conjugacy class is Haar null possibly except for 
  a countable set of classes, namely the conjugacy classes of the multishifts. 
    \end{corollary}
    
    \begin{question}
      \label{q:multishifts haar null}
      Are all conjugacy classes in $\mathcal{U}(\ell^2)$ Haar null? Equivalently, are all the conjugacy classes of the multishifts in $\mathcal{U}(\ell^2)$ Haar null?
    \end{question}

\section{An application: Haar null-meagre decompositions}

In this section we use our results to answer Question \ref{q:decomposition} in some important special cases. 
The following is an easy corollary of Theorem \ref{t:homeo [0,1] conjugacy classes}. 

\begin{corollary}
  \label{c:homeo [0,1] is Haar null cup meagre}
  The group $\hp([0, 1])$ can be partitioned into a Haar null and a meagre set.
\end{corollary}
\begin{proof}
  Let $\mathcal{N} = \{ f \in \hp([0, 1]) : |\fix(f)| > \aleph_0\}$ and 
  let $\mathcal{M} = \{ f \in \hp([0, 1]) : |\fix(f)| \le \aleph_0\}$. Clearly 
  $\mathcal{M} \cup \mathcal{N}$ is a partition of $\hp([0, 1])$, and using 
  Theorem \ref{t:homeo [0,1] conjugacy classes}, $\mathcal{N}$ is Haar 
  null, while using the well-known fact that the generic homeomorphism has a Cantor set 
  of fixed points (see e.g. \cite{GW} for an even stronger statement), $\mathcal{M}$ is meagre.
\end{proof}

Now we use Theorem \ref{t:haar positive conj classes in homeo s1} to show the following.

\begin{corollary}
  \label{c:homeo S^1 is Haar null cup meagre}
  The group $\hp(\mathbb{S}^1)$ can be partitioned into a Haar null and a meagre set.
\end{corollary}
\begin{proof}
  Similarly to the proof of the previous corollary, let 
  $$\mathcal{N} = \{ f \in \hp(\mathcal{S}^1) : \text{ the number of periodic points of $f$ is 
  uncountable}\}$$ and let $$\mathcal{M} = \{ f \in \hp(\mathcal{S}^1) : \text{ the number of 
  periodic points of $f$ is countable}\}.$$ Clearly 
  $\mathcal{M} \cup \mathcal{N}$ is a partition of $\hp([0, 1])$, and using 
  Theorem \ref{t:haar positive conj classes in homeo s1}, $\mathcal{N}$ is Haar 
  null, while using the fact that for the generic $f \in \hp(\mathcal{S}^1)$ 
  the set of periodic points of $f$ is homeomorphic to the Cantor set 
  (see \cite[Theorem 3]{CHPZ}), $\mathcal{M}$ is meagre.
\end{proof}

The following proposition is an easy consequence of a result of Dougherty \cite{D}.

\begin{proposition}
  \label{p:factor -> original decomposition} 
  Let $G$ and $H$ be Polish groups and suppose that there exists a continuous, surjective homomorphism $\phi : G \to H$. If $H$ can be partitioned into a Haar null and a meagre set then $G$ can be partitioned in such a way as well. 
\end{proposition}
\begin{proof}
  Let $H = N \cup M$, where $N$ is Haar null and $M$ is meagre. By \cite[Proposition 8]{D}, the inverse image of a Haar null set under a continuous, surjective homomorphism is also Haar null, hence $\phi^{-1}(N)$ is Haar null in $G$. To see that $\phi^{-1}(M)$ is meagre, note that $M \subset \bigcup_{n \in \N} S_n$ where each $S_n$ is closed, nowhere dense. Hence, for each $n$, $\phi^{-1}(S_n)$ is closed using that $\phi$ is continuous. Since $\phi$ is a continuous, surjective homomorphism, $\phi$ is open by \cite[Theorem 2.3.3]{Gaobook}. Therefore, $\phi^{-1}(S_n)$ is nowhere dense in $G$ for each $n$. Thus $G = \phi^{-1}(N) \cup \phi^{-1}(M)$ is a Haar null-meagre decomposition of $G$. 
\end{proof}
\begin{corollary}
  Let $G$ be a Polish group and $H \le G$ a closed normal subgroup. Then $G / H$ is a Polish group, and if $G / H$ can be partitioned into a Haar null and a meagre set then $G$ can be partitioned in such a way as well.  
\end{corollary}
\begin{proof}
  For the well-known facts that $G / H$ is a Polish group and the canonical projection from $G$ to $G / H$ satisfies the conditions of Proposition \ref{p:factor -> original decomposition} see e.g.~Proposition 1.2.3 of \cite{becker1996descriptive} and the preceding remarks. Then an application of Proposition \ref{p:factor -> original decomposition} completes the proof. 
\end{proof}

We use Proposition \ref{p:factor -> original decomposition} to obtain a corollary of Corollary \ref{c:homeo S^1 is Haar null cup meagre}. 
\begin{corollary}
  The group $\hp(\mathbb{D}^2)$ of orientation preserving homeomorphisms of the disc $\mathbb{D}^2$ can be partitioned into a Haar null and a meagre set. 
\end{corollary}
\begin{proof}
  It is not hard to see that the restriction of every homeomorphism $f \in \hp(\mathbb{D}^2)$ to the boundary of $\mathbb{D}^2$ is an element of $\hp(\mathbb{S}^1)$. Then it is easy to see that this map is a continuous, open and surjective homomorphism, hence the proof is complete using Proposition \ref{p:factor -> original decomposition}.
\end{proof}

The following questions, however, remain open.

\begin{question}
  \label{q:D^n, S^n decomp}
  Is it true that $\hp(\mathbb{D}^n)$ (for $n \ge 3$) and $\hp(\mathbb{S}^n)$ (for $n \ge 2$) can be partitioned into a Haar null and a meagre set?
\end{question}

\begin{question}
  \label{q:[0,1]^N decomp}
  Is it true that the homeomorphism group of the Hilbert cube $[0, 1]^\N$ can be partitioned into a Haar null and a meagre set?
\end{question}

\section{Open problems}

In this section we collect the open problems of the paper. First of all, it would be interesting to obtain general results about homeomorphism groups.

\begin{question}
  What can we say about the size of conjugacy classes of the group of homeomorphisms of a compact metric space?
\end{question}

The following special case is particularly important.

\begin{question}
Which conjugacy classes are non-Haar null in the group of homeomorphisms of the Cantor set?
\end{question}

Concerning the groups we investigated in this paper, the following two problems remain open. 

\begin{namedthm*}{Question \ref{q:S1 irrational rot number}}
  Is the set $\{f \in \hp(\mathbb{S}^1) : \tau(f) \not \in \Q\}$ Haar null?
\end{namedthm*}
\begin{namedthm*}{Question \ref{q:multishifts haar null}}
  Are all conjugacy classes in $\mathcal{U}(\ell^2)$ Haar null? Equivalently, are all the conjugacy classes of the multishifts in $\mathcal{U}(\ell^2)$ Haar null?
\end{namedthm*}

As an application of our results, we answered special cases of the following question that remains open in general. 

\begin{namedthm*}{Question \ref{q:decomposition}}
  Suppose that $G$ is an uncountable Polish group. Can it be written as the 
  union of a meagre and a Haar null set?
\end{namedthm*}

In particular, we are interested in the following special cases.

\begin{namedthm*}{Question \ref{q:D^n, S^n decomp}}
  Is it true that $\hp(\mathbb{D}^n)$ (for $n \ge 3$) and $\hp(\mathbb{S}^n)$ (for $n \ge 2$) can be partitioned into a Haar null and a meagre set?
\end{namedthm*}
\begin{namedthm*}{Question \ref{q:[0,1]^N decomp}}
  Is it true that the homeomorphism group of the Hilbert cube $[0, 1]^\N$ can be partitioned into a Haar null and a meagre set?
\end{namedthm*}

\textbf{Acknowledgements.} We would like to thank to R. Balka, Z. Gyenis, A. Kechris, C. Rosendal, S. Solecki and P. Wesolek for many valuable remarks and discussions.

\bibliographystyle{abbrv}
\bibliography{ran}

\end{document}